\documentclass[aos,MSNbibl,seceqn,dvips]{arximspdf}
\usepackage{mathbh}
\usepackage{graphicx}

\doi{10.1214/14-AOS1224} 
\volume{42}
\issue{4}
\pubyear{2014}
\firstpage{1312}
\lastpage{1346}

\makeatletter
\newcommand{\rrVert}{\Vert}
\newcommand{\rrvert}{\vert}
\newcommand{\llVert}{\Vert}
\newcommand{\llvert}{\vert}
\def\cal{\mathcal}
\def\mathds{\mathbb}
\newcommand{\eqref}[1]{(\ref{#1})}
\newtheorem{thmm}{Theorem}[section]
\newproclaim{defi}[thmm]{Definition}
\newproclaim{remark}[thmm]{Remark}

\newtheorem{lem}[thmm]{Lemma}
\newtheorem{cor}[thmm]{Corollary}
\newproclaim{assump}[thmm]{Assumption}
\renewcommand{\:}{\mathrel{\mathop{:}}}
\newcommand{\E}{\mathbb{E}}
\newcommand{\N}{\mathbb{N}}
\newcommand{\R}{\mathbb{R}}
\newcommand{\eps}{\varepsilon}
\renewcommand{\epsilon}{\varepsilon}
\newcommand{\KLEINO}{\scriptstyle{\mathcal{O}}}
\newcommand{\1}{\mathbh{1}}
\newcommand{\AVAR}{\operatorname{AVAR}}
\newcommand{\Var}{\operatorname{\mathbb{V}ar}}
\newcommand{\Cov}{\operatorname{\mathbb{C}OV}}
\newcommand{\trace}{\operatorname{trace}}
\newcommand{\diag}{\operatorname{diag}}
\newcommand{\Id}{\mathrm{Id}}
\renewcommand{\phi}{\varphi}
\renewcommand{\theta}{\vartheta}
\renewcommand{\subset}{\subseteq}
\makeatother

\begin{document}
\begin{frontmatter}

\title{Estimating the quadratic covariation matrix from noisy
observations: Local method of moments and~efficiency\thanksref{T2}}
\runtitle{Efficient quadratic covariation matrix estimation}
\thankstext{T2}{Supported by the Deutsche
Forschungsgemeinschaft via SFB 649 \textit{\"Okonomisches Risiko}
and FOR 1735 \textit{Structural Inference in Statistics: Adaptation and
Efficiency}.}

\begin{aug}
\author[a]{\fnms{Markus}~\snm{Bibinger}\ead[label=e1]{bibinger@math.hu-berlin.de}},
\author[b]{\fnms{Nikolaus}~\snm{Hautsch}\thanksref{t3}\ead[label=e2]{nikolaus.hautsch@univie.ac.at}},
\author[c]{\fnms{Peter}~\snm{Malec}\ead[label=e3]{malecpet@hu-berlin.de}}
\and
\author[a]{\fnms{Markus}~\snm{Rei\ss}\corref{}\ead[label=e4]{mreiss@math.hu-berlin.de}}
\runauthor{Bibinger, Hautsch, Malec and Rei{\ss}}
\thankstext{t3}{Supported by the Wiener Wissenschafts-,
Forschungs- und Technologiefonds (WWTF).}
\affiliation{Humboldt-Universit\"at zu Berlin, University of Vienna,
Humboldt-Universit\"at zu Berlin and Humboldt-Universit\"at zu Berlin}
\address[a]{M. Bibinger\\
M. Rei{ss}\\
Institut f\"ur Mathematik\\
Humboldt-Universit\"at zu Berlin\\
Unter den Linden 6\\
10099 Berlin\\
Germany\\
\printead{e1}\\
\phantom{E-mail:\ }\printead*{e4}}
\address[b]{N. Hautsch\\
Department of Statistics\\
\quad and Operations Research\\
University of Vienna\\
Oskar-Morgenstern-Platz 1\\
1090 Vienna\\
Austria\\
\printead{e2}}
\address[c]{P. Malec\\
School of Business and Economics\\
Humboldt-Universit\"at zu Berlin\\
Spandauer Str. 1\\
10178 Berlin\\
Germany\\
\printead{e3}}
\end{aug}

\received{\smonth{11} \syear{2013}}
\revised{\smonth{4} \syear{2014}}

%
\begin{abstract}
An efficient estimator is constructed for the quadratic covariation or
integrated co-volatility matrix of a multivariate continuous martingale
based on noisy and nonsynchronous observations under high-frequency
asymptotics. Our approach relies on an asymptotically equivalent
continuous-time observation model where a local generalised method of
moments in the spectral domain turns out to be optimal. Asymptotic
semi-parametric efficiency is established in the Cram\'er--Rao sense.
Main findings are that nonsynchronicity of observation times has no
impact on the asymptotics and that major efficiency gains are possible
under correlation. Simulations illustrate the finite-sample behaviour.
\end{abstract}

%
\end{frontmatter}

\section{Introduction}

We study the estimation of the quadratic covariation (or integrated
co-volatility) matrix of a multi-dimensional continuous
semi-martingale. Semi-martingales are central objects in stochastics
and the estimation of their quadratic covariation from noisy
observations is certainly a fundamental topic on its own. Because of
its key importance in finance, this question attracts high attention
from high-frequency financial statistics with implications for
portfolio allocation, risk quantification, hedging or asset pricing.
While the univariate case has been studied extensively from both angles
(see, e.g., the survey of Andersen {et al.}~\cite{andboldie2010} or
recent work by Rei{ss} \cite{reiss} and Jacod and Rosenbaum \cite
{jacodrosenbaum}), statistical inference for the quadratic covariation
matrix is not yet well understood. This is, on the one hand, due to a
richer geometry, for example, induced by noncommuting matrices,
generating new effects and calling for a deeper mathematical
understanding. On the other hand, statistical challenges arise by the
use of underlying multivariate high-frequency data which are typically
polluted by noise. Though they open up new ways for statistical
inference, their noise properties, significantly different sample sizes
(induced by different trading frequencies) as well as irregular and
asynchronous spacing in time make estimation in these models far from
obvious. Different approaches exist, partly furnish unexpected results,
but are rather linked to the method than to the statistical problem. In
this paper, we strive for a general understanding of the statistical
problem itself, in particular the question of efficiency, while at the
same time we develop a local method of moments approach which yields a
simple and efficient estimator.

To remain concise, we consider the basic statistical model where the
$d$-dimensional discrete-time process
{\renewcommand{\theequation}{$\mathcal{E}_0$}
\begin{equation}
\label{E0}
Y_i^{(l)}=X_{t_i^{(l)}}^{(l)}+
\eps_i^{(l)},\qquad 0\le i\le n_l,1\le l\le d,
\end{equation}}
\hspace*{-3pt}is observed with the $d$-dimensional continuous martingale
\[
X_t=X_0+\int_0^t
\Sigma^{1/2}(s) \,d{B}_s, \qquad t\in[0,1],
\]
in terms of a $d$-dimensional standard Brownian motion ${B}$ and the
squared (instantaneous or spot) co-volatility matrix
\[
\Sigma(t)= \bigl(\Sigma_{lr}(t) \bigr)_{1\le l,r\le d}\in\R
^{d\times d}.
\]
In financial applications, $X_t$ corresponds to the multi-dimensional
process of fundamental asset prices whose martingale property complies
with market efficiency and exclusion of arbitrage. The major quantity
of interest is the quadratic covariation matrix $\int_0^1 \Sigma(t)
\,dt$, computed over a normalised interval such as, for example, a
trading day.

The signal part $X$ is assumed to be independent of the observation
errors $(\mathbf{\eps}_i^{(l)}), 1\le l\le d,1\le i\le n_l$, which are
mutually independent and centered normal with variances $\eta_l^2$. In
the literature on financial high-frequency data, these errors capture
microstructure frictions in the market (\textit{microstructure noise}).
The observation times are given via quantile transformations as
$t_i^{(l)}=F_l^{-1}(i/n_l)$ for some distribution functions $F_l$.
While the model \eqref{E0} is certainly an idealisation of many real
data situations, its precise analysis delivers a profound understanding
and thus serves as a basis for developing procedures in more complex
models. During the revision of this paper, Altmeyer and Bibinger \cite
{stable} have shown that the local method of moments in a general
continuous semi-martingale model (including drift and stochastic
volatility) and under general moment conditions on the noise $(\eps
_i^{(l)})$ enjoys similar asymptotic properties as in our basic model.
In particular, a stable central limit theorem is established. A similar
extension to random and endogenous observations times $(t_i^{(l)})$
would be of high interest, but does not seem obvious; see Li {et
al.} \cite{Lietal} for recent work on the case without noise and some
empirical evidence for endogenous times.

Estimation of the quadratic covariation of a price process is a core
research topic in current financial econometrics and various approaches
have been put forward in the literature. The realised covariance
estimator was studied by Barndorff-Nielsen and Shephard \cite{bns04}
for a setting that neglects both microstructure noise and effects due
to the nonsyncronicity of observations. Hayashi and Yoshida \cite{hy}
propose an estimator which is efficient under the presence of
asynchronicity, but without noise. Methods accounting for both types of
frictions are the quasi-maximum-likelihood approach by A\"{i}t-Sahalia
{et al.} \cite{aitfanxiu2010}, realised kernels by
Barndorff-Nielsen {et al.} \cite{bn2011}, pre-averaging by
Christensen {et al.} \cite{kinnepoldivet}, the two-scale estimator
by Zhang~\cite{zhang11} and the local spectral estimator by Bibinger
and Rei{ss} \cite{bibingerreiss}. In contrast to the univariate case,
the asymptotic properties of these estimators are involved and the
structure of the terms in the asymptotic variance deviate
significantly. None of the methods outperforms the others for all
settings, calling for a lower efficiency bound as a benchmark.

In this paper, we propose a local method of moments (LMM) estimator,
which is optimal in a semi-parametric Cram\'er--Rao sense under the
presence of noise and the nonsynchronicity of observations. The idea
rests on the (strong) asymptotic equivalence in Le Cam's sense of model
\eqref{E0} with the continuous time signal-in-white-noise model
{\renewcommand{\theequation}{$\mathcal{E}_1$}
\begin{equation}
\label{E1}
\,d Y_t=X_t \,dt+\diag
\bigl(H_{n,l}(t)\bigr)_{1\le l\le d} \,d{W}_t,\qquad t\in[0,1],
\end{equation}}
\hspace*{-3pt}where ${W}$ is a standard $d$-dimensional Brownian motion independent
of ${B}$ and the component-wise local noise level is
\setcounter{equation}{0}
\begin{equation}
\label{noiselevel}H_{n,l}(t):=\eta _l\bigl(n_lF_l'(t)
\bigr)^{-{1}/{2}}.
\end{equation}
Here, $F_l'(t)$ represents the local frequency of occurrences
(``observation density'') and thus $n_lF_l'(t)$ corresponds to the
local sample size, which is the continuous-time analogue of the so
called \textit{quadratic variation of time}, discussed in the literature.
The advantage of the continuous-time model \eqref{E1} is particularly
distinctive in the multivariate setting where asynchronicity and
different sample sizes in the discrete data \eqref{E0} blur the
fundamental statistical structure. If two sequences of statistical
experiments are asymptotically equivalent, then any statistical
procedure in one experiment has a counterpart in the other experiment
with the same asymptotic properties; see Le Cam and Yang \cite
{lecamyang} for details. Our equivalence proof is constructive such
that the procedure we shall develop for \eqref{E1} has a concrete
equivalent in \eqref{E0} with the same asymptotic properties.

A remarkable theoretical consequence of the equivalence between \eqref
{E0} and \eqref{E1} is that under noise, the asynchronicity of the data
does not affect the asymptotically efficient procedures. In fact, in
model \eqref{E1}, the distribution functions $F_l$ only generate
time-varying local noise levels $H_{n,l}(t)$, but the shift between
observation times of the different processes does not matter. Hence,
locally varying observation frequencies have the same effect as locally
varying variances of observation errors and may be pooled. This is in
sharp contrast to the noiseless setting where the variance of the
Hayashi--Yoshida estimator \cite{hy} suffers from errors due to
asynchronicity, which carries over to the pre-averaged version by
Christensen {et al.} \cite{kinnepoldivet} designed for the noisy
case. Only if the noise level is assumed to tend to zero so fast that
the noiseless case is asymptotically dominant, then the
nonsynchronicity may induce additional errors.

Our proposed estimator builds on a locally constant approximation of
the continuous-time model \eqref{E1} with equi-distant blocks across
all dimensions. We show that the errors induced by this approximation
vanish asymptotically. Empirical local Fourier coefficients allow for a
simple moment estimator for the block-wise spot co-volatility matrix.
The final estimator then corresponds to a generalised method of moments
estimator of $\int_0^1 \Sigma(t) \,dt$, computed as a weighted sum of all
individual local estimators (across spectral frequencies and time).
Asymptotic efficiency of the resulting LMM estimator is shown to be
achieved by an optimal weighting scheme based on the Fisher information
matrices of the underlying local moment estimators.

As a result of the noncommutativity of the Fisher information
matrices, the LMM estimator for one element of the covariation matrix
generally depends on \textit{all} entries of the underlying local
covariances. Consequently, the volatility estimator in one dimension
substantially gains in efficiency when using data of all other
potentially correlated processes. These efficiency gains in the
multi-dimensional setup constitute a fundamental difference to the case
of i.i.d. observations of a Gaussian vector where the empirical
variance of one component is an efficient estimator. Here, using the
other entries cannot improve the variance estimator unless the
correlation is known; cf. the classical Example 6.6.4 in Lehmann and
Casella \cite{LehmannCasella}.
This finding is natural for covariance estimation under nonhomogeneous
noise and because of its general interest we shall discuss a related
i.i.d. example in Section~\ref{sec:1b}. The possibility of efficiency
gainshas been known in specific cases for quite a while, which was then
also discussed in Shephard and Xiu \cite{xiu} and Liu and Tang \cite
{qmle}, but until now a general view and a precise lower bound were missing.

The next Section~\ref{sec:1b} gives an overview of the estimation
methodology and explains the major implications in a compact and
intuitive way with the subsequent sections establishing the general
results in full rigour. Emphasis is put on the concrete form of the
efficient asymptotic variance-covariance structure which provides a
rich geometry and has surprising consequences in practice.

In Section~\ref{sec:2}, we establish the asymptotic equivalence in Le
Cam's sense of models \eqref{E0} and \eqref{E1} in Theorem~\ref{theo1}.
The regularity assumptions required for $\Sigma$ are less restrictive
than in Rei{ss} \cite{reiss} and particularly allow $\Sigma$ to jump.

Section~\ref{sec:3} introduces the LMM estimator in the spectral
domain. Theorem~\ref{cltorlle} provides a multivariate central limit
theorem (CLT) for an oracle LMM estimator, using the unknown optimal
weights and an information-type matrix for normalisation, which allows
for asymptotically diverging sample sizes in the coordinates.
Specifying to sample sizes of the same order $n$, Corollary~\ref
{corclt} yields a CLT with rate $n^{1/4}$ and a covariance structure
between matrix entries, which is explicitly given by concise matrix
algebra. Then pre-estimated weight matrices generate a fully adaptive
version of the LMM-estimator, which by Theorem~\ref{cltadlle} shares
the same asymptotic properties as the oracle estimator. This allows
intrinsically feasible confidence sets without pre-estimating
asymptotic quantities.

In Section~\ref{sec:4}, we show that the asymptotic covariance matrix
of the LMM estimator attains a lower bound in the Cram\'er--Rao sense.
This lower bound is achieved by a combination of space--time
transformations and advanced calculus for covariance operators.
Detailed proofs are given in the supplementary file \cite{supplement}.

Finally, the discretisation and implementation of the estimator for
model \eqref{E0} is briefly described in Section~\ref{sec:5} and
presented together with some numerical results. We apply the method for
a complex and realistic simulation scenario, obtained by a
superposition of time-varying seasonality functions, calibrated to real
data, and a semi-martingale process with stochastic volatilities
exhibiting leverage effects. The observation times are asynchronous and
random. We conclude that the finite sample behaviour of the LMM
estimators is well predicted by the asymptotic theory (even in cases
where a formal proof lacks). Some comparison with competing procedures
is provided.


\section{Principles and major implications}
\label{sec:1b}

\subsection{Spectral LMM methodology}

The time interval $[0,1]$ is partitioned into small blocks
$[kh,(k+1)h)$, $k=0,\ldots,h^{-1}-1$, such that on each block a
constant parametric co-volatility matrix estimate can be sought for (cf.
the \textit{local-likelihood} approach).
The main estimation idea is then to use block-wise spectral statistics
$(S_{jk})$, which represent localised Fourier coefficients as in Reiss
{} \cite{reiss}. Specifying to the original discrete data \eqref{E0},
they are calculated as
%
\begin{equation}
\label{specdis} S_{jk}=\pi jh^{-1} \Biggl(\sum
_{\nu=1}^{n_l} ( Y_{\nu}-Y_{\nu
-1} )
\Phi_{jk} \biggl(\frac{t_{\nu-1}^{(l)}+t_{\nu}^{(l)}}{2} \biggr) \Biggr)_{1\le l\le
d}\in
\R^d,
\end{equation}
with sine functions $\Phi_{jk}$ of frequency index $j$ on each block
$[kh,(k+1)h]$ given by
%
\begin{equation}
\label{EqPhijk} \Phi_{jk}(t)=\frac{\sqrt{2h}}{j\pi} \sin{ \bigl(j\pi
h^{-1} (t-kh ) \bigr)}\1_{[kh,(k+1)h]}(t),\qquad j\ge1.
\end{equation}
The same blocks are used across all dimensions $d$ with their size $h$
being determined by the least frequently observed process.

The statistics $(S_{jk})$ are Riemann--Stieltjes sum approximations to
Fourier integrals based on a possibly nonequidistant grid. The
discrete-time processes $(Y_i^{(l)})$ can be transformed into a
continuous-time process via linear interpolation in each dimension,
which yields piecewise constant (weak) derivatives, with the $S_{jk}$
being interpreted as integrals over these derivatives. Mathematically,
the asymptotic equivalence of \eqref{E0} and \eqref{E1} based on this
linear interpolation is made rigorous in Theorem~\ref{theo1}. The
required regularity condition is that $\Sigma(t)$ is the sum of an
$L^2$-Sobolev function of regularity $\beta$ and an $L^2$-martingale
and the size of $\beta$ accommodates for asymptotically separating
sample sizes $(n_l)_{1\le l\le d}$. In model \eqref{E1} by partial
integration, the statistics $S_{jk}$ then correspond to
%
\begin{equation}
\label{spec}S_{jk}^{(l)}=\pi jh^{-1}\int
_{kh}^{(k+1)h}\varphi_{jk}(t)
\,dY^{(l)}(t)
\end{equation}
with block-wise cosine functions $\phi_{jk}=\Phi_{jk}'$ which form an
orthonormal system in $L^2([0,1])$. As they serve also as the
eigenfunctions of the Karhunen--Lo\`eve decomposition of a Brownian
motion, they carry maximal information for $\Sigma$. What is more, the
spectral statistics $S_{jk}$ de-correlate the observations, and thus
form their (block-wise) principal components, assuming that $\Sigma$
and the noise levels are block-wise constant. Then
the entire family $(S_{jk})_{jk}$ is independent and
%
\begin{equation}
\label{localnorm2} S_{jk}\sim\mathbf{N}(0,C_{jk}),\qquad
C_{jk}=\Sigma^{kh}+\pi ^2j^2h^{-2}
\diag \bigl(H^{kh}_{n,l}\bigr)_l^2,
\end{equation}
with the $k$th block average $\Sigma^{kh}$ of $\Sigma$ and $H_{n,l}^{kh}$
encoding the local noise level; cf.~\eqref{localnorm} below.

This relationship suggests to estimate $\Sigma^{kh}$ in each frequency
$j$ by bias-corrected spectral covariance matrices $S_{jk}S_{jk}^{\top}
-\pi^2j^2h^{-2}\diag{((H_{n,l}^{kh})^2)}_l$. The resulting \textit{local
method of moment (LMM) estimator} then takes weighted sums across all
frequencies and blocks
\[
\operatorname{LMM}^{(n)}:=\sum_{k=0}^{h^{-1}-1}
h \sum_{j=1}^{\infty}W_{jk} \operatorname{vec}
\bigl(S_{jk}S_{jk}^{\top} -\pi^2j^2h^{-2}
\diag{\bigl(\bigl(H_{n,l}^{kh}\bigr)^2
\bigr)}_l \bigr),
\]
where $W_{jk}\in\R^{d^2\times d^2}$ are weight matrices and matrices
$A\in\R^{d\times d}$ are transformed into vectors via
\[
\operatorname{vec}(A):= (A_{11},A_{21},\ldots,A_{d1},A_{12},A_{22},
\ldots,A_{d2}, \ldots,A_{d(d-1)},A_{dd}
)^{\top}\in{\mathds{R}}^{d^2}.
\]
To ensure efficiency, the oracle and adaptive choice of the weight
matrices $W_{jk}$ are based on Fisher information calculus; see
Section~\ref{sec:3} below. Let us mention that scalar weights for each matrix
estimator entry as in Bibinger and Rei{ss} \cite{bibingerreiss} will
not be sufficient to achieve (asymptotic) efficiency and the $W_{jk}$
will be densely populated.

The matrix estimator {per se} is not ensured to be positive
semi-definite, but it is symmetric and can be projected onto the cone
of positive semi-definite matrices by putting negative eigenvalues to
zero. This projection only improves the estimator, while the adjustment
is asymptotically negligible in the CLT. For the relevant question of
confidence sets, the estimated nonasymptotic Fisher information
matrices are positive--semi-definite (basically, estimating $C_{jk}$
from above) and finite sample inference is always feasible.

\subsection{The efficiency bound}

Deriving the covariance structure of a matrix estimator requires tensor
notation; see, for example, Fackler \cite{fackler} or textbooks on
multivariate analysis. Kronecker products $A\otimes B\in\R^{d^2\times
d^2}$ for $A,B\in\R^{d\times d}$ are defined as
\[
(A\otimes B)_{d(p-1)+q,d(p'-1)+q'}=A_{pp'}B_{qq'},\qquad
p,q,p',q'=1,\ldots,d.
\]
The covariance structure for the empirical covariance matrix of a
standard Gaussian vector is defined as
%
\begin{equation}
\label{EqZ} {\cal Z}=\Cov\bigl(\operatorname{vec}\bigl(ZZ^\top\bigr)\bigr)\in
\R^{d^2\times d^2}\qquad\mbox{for }Z\sim \mathbf{N}(0,E_d).
\end{equation}
We can calculate $\cal Z$ explicitly as
\[
{\cal Z}_{d(p-1)+q,d(p'-1)+q'}=(1+\delta_{p,q})\delta_{\{p,q\},\{
p',q'\}
},\qquad
p,q,p',q'=1,\ldots,d,
\]
exploiting the property ${\cal Z}\operatorname{vec}(A)=\operatorname{vec}(A+A^\top)$ for all $A\in
\R
^{d\times d}$.
It is classical (cf. Lehmann and Casella \cite{LehmannCasella}), that
for $n$ i.i.d. Gaussian observations $Z_i\sim{\mathbf N}(0,\Sigma)$,
the empirical covariance matrix $\hat\Sigma_n=\frac{1}n\sum_{i=1}^nZ_iZ_i^\top$ is an asymptotically efficient estimator of
$\Sigma
$ satisfying
\[
\sqrt{n} \operatorname{vec}(\hat\Sigma_n-\Sigma)\mathop{\rightarrow}\limits^{{\cal L}} \mathbf {N}
\bigl(0,(\Sigma\otimes\Sigma){\cal Z}\bigr).
\]
The asymptotic variance can be easily checked by the rule
$\operatorname{vec}(ABC)=(C^\top\otimes A)\operatorname{vec}(B)$ and the fact that $\cal Z$ commutes
with $(\Sigma\otimes\Sigma)^{1/2}=\Sigma^{1/2}\otimes\Sigma
^{1/2}$ such
that $\Cov(\operatorname{vec}(\hat\Sigma_n))$ equals
\[
\Cov\bigl(\operatorname{vec}\bigl(\Sigma^{1/2}ZZ^\top\Sigma^{1/2}
\bigr)\bigr) =\bigl(\Sigma^{1/2}\otimes\Sigma^{1/2}\bigr){\cal
Z}\bigl(\Sigma^{1/2}\otimes \Sigma ^{1/2}\bigr)=(\Sigma\otimes
\Sigma){\cal Z}.
\]

Before proceeding, let us provide an intuitive understanding of the
efficiency gains from other dimensions by looking at another easy case
with independent observations. Suppose an i.i.d. sample $Z_1,\ldots,Z_n\sim{\mathbf N}(0,\Sigma)$, $\Sigma\in\R^{d\times d}$ unknown, is
observed indirectly via $Y_j=Z_j+\eps_j$, blurred by independent
nonhomogeneous noise $\eps_j\sim{\mathbf N}(0,\eta_j^2E_d)$,
$j=1,\ldots,n$, with identity matrix $E_d$ and $\eta_1,\ldots,\eta_n>0$
known. Then the sample covariance matrix $\hat C_Y=\sum_{j=1}^nY_jY_j^\top$ and a bias correction yields a first natural
estimator $\hat\Sigma^{(1)}=\hat C_Y-\eta^2 E_d$, $\eta^2=\sum_j\eta
_j^2/n$. Yet, we can weight each observation differently by some
$w_j\in
\R$ with $\sum_jw_j=1$ and obtain a second estimator $\hat\Sigma
^{(2)}=\sum_{j=1}^nw_j(Y_jY_j^\top-\eta_j^2E_d)$. For optimal
estimation of the first variance $\Sigma_{11}$, we should choose (as in
a weighted least squares approach) $w_j=(\Sigma_{11}+\eta
_j^2)^{-2}/(\sum_i(\Sigma_{11}+\eta_i^2)^{-2})$ to obtain
\[
\Var\bigl(\hat\Sigma_{11}^{(2)}\bigr)=2 \Biggl(\sum
_{j=1}^n\bigl(\Sigma_{11}+\eta
_j^2\bigr)^{-2} \Biggr)^{-1}\le
\frac{2}{n^2}\sum_{j=1}^n\bigl(
\Sigma_{11}+\eta _j^2\bigr)^2=
\Var\bigl(\hat\Sigma_{11}^{(1)}\bigr),
\]
where the bound is due to Jensen's inequality.
More generally, we can use weight matrices $W_j\in\R^{d^2\times d^2}$
and introduce
$\hat\Sigma^{(3)}=\sum_{j=1}^nW_j \operatorname{vec}(Y_jY_j^\top-\eta_j^2E_d)$. Since
the matrices $C_j=\Sigma+\eta_j^2E_d$ commute, its covariance structure
is given by
$\Cov(\hat\Sigma^{(3)})=\sum_{j=1}^n W_j(C_j\otimes C_j){\cal Z}
W_j^\top$.
This is minimal for $W_j=(\sum_i C_i^{-1}\otimes
C_i^{-1})^{-1}(C_j^{-1}\otimes C_j^{-1})$, which gives
$\Cov(\hat\Sigma^{(3)})=(\sum_j C_j^{-1}\otimes C_j^{-1})^{-1}{\cal Z}$.
The matrices $W_j$ are diagonal if all $\eta_j$ coincide or if $\Sigma$
is diagonal. Otherwise, the estimator for one matrix entry involves in
general all other entries in $Y_jY_j^\top$ and in particular $\Var
(\hat
\Sigma^{(3)}_{11})<\Var(\hat\Sigma^{(2)}_{11})$ holds. Considering as
$(Y_j)_{j\ge1}$ the spectral statistics $(S_{jk})_{j\ge1}$ on a
fixed block $k$, this example reveals the heart of our analysis for the
LMM estimator.

Similar to the i.i.d. case, for equidistant observations
$(X_{i/n})_{1\le i\le n}$ of $X_t=\int_0^t\Sigma(s)\,dB_s$ without noise,
the realised covariation matrix
\[
\widehat{\mathit{RCV}}_n=\sum_{i=1}^n(X_{i/n}-X_{(i-1)/n})
(X_{i/n}-X_{(i-1)/n})^\top
\]
satisfies the $d^2$-dimensional central limit theorem
\[
\sqrt{n} \operatorname{vec} \biggl(\widehat{\mathit{RCV}}_n-\int_0^1
\Sigma(t) \,dt \biggr)\mathop{\rightarrow}\limits^{{\cal L}} \mathbf{N} \biggl(0, \biggl(\int
_0^1\Sigma(t)\otimes\Sigma(t) \,dt \biggr){\cal
Z} \biggr),
\]
provided $t\mapsto\Sigma(t)$ is Riemann-integrable. In the
one-dimensional case, it is known that in the presence of noise the
optimal rate of convergence not only changes from $n^{-1/2}$ to
$n^{-1/4}$, but also the optimal variance changes from $2\sigma^4$ to
$8\sigma^3$. The corresponding analogue of $(\Sigma\otimes\Sigma
){\cal
Z}$ in the noisy case is not obvious at all. So far, only the result by
Barndorff-Nielsen et al. \cite{bn2011}, establishing $(\Sigma\otimes
\Sigma){\cal Z}$ as limiting variance under the suboptimal rate
$n^{-1/5}$, was available and even a conjecture concerning the
efficiency bound was lacking.

To illustrate our multivariate efficiency results under noise let us
for simplicity illustrate a special case of Corollary~\ref{corclt} for
equidistant observations, that is, $t_i^{(l)}=i/n$, and homogeneous noise
level $\eta_l=\eta$. Then the oracle (and also the adaptive) estimator
$\operatorname{LMM}^{(n)}$ satisfies under mild regularity conditions
(omitting the integration variable $t$)
\[
n^{1/4} \biggl(\operatorname{LMM}^{(n)} -\int
_0^1\operatorname{vec}(\Sigma) \biggr)\mathop{\rightarrow}\limits^{{\cal L}} \mathbf{N} \biggl(0,2\eta \int_0^1
\bigl(\Sigma\otimes\Sigma^{1/2} + \Sigma^{1/2} \otimes \Sigma
\bigr){\cal Z} \biggr).
\]
In Theorem~\ref{ThmCR}, it will be shown that this asymptotic
covariance structure is optimal in a semi-parametric Cram\'er--Rao
sense. Consequently, the efficient asymptotic variance $\AVAR$ for
estimating $\int_0^1\Sigma_{pp}(t) \,dt$ is
\[
\AVAR \biggl(\int_0^1\Sigma_{pp}(t)
\,dt \biggr)=8\eta\int_0^1 \Sigma
_{pp}(t) \bigl(\Sigma^{1/2}(t)\bigr)_{pp} \,dt.
\]
For the asymptotic variance of the estimator of $\int_0^1\Sigma_{pq}(t)
\,dt$, we obtain
\[
2\eta\int_0^1 \bigl(\bigl(\Sigma
^{1/2}\bigr)_{pp}\Sigma_{qq}+ \bigl(
\Sigma^{1/2}\bigr)_{qq}\Sigma_{pp} + 2\bigl(\Sigma
^{1/2}\bigr)_{pq}\Sigma_{pq} \bigr) (t) \,dt.
\]
Let us illustrate specific examples. First, in the case $d=1$ and
$\Sigma=\sigma^2$, the asymptotic variance simplifies to
\[
\AVAR \biggl(\int_0^1\sigma^2(t)
\,dt \biggr)=8\eta\int_0^1\sigma^3(t)
\,dt,
\]
coinciding with the efficiency bound in Rei{ss} \cite{reiss}. For
$d>1$, $p\neq q$ in the independent case $\Sigma=\diag(\sigma
_p^2)_{1\le p\le d}$, we find
\[
\AVAR \biggl(\int_0^1\Sigma_{pq}(t)
\,dt \biggr) =2\eta\int_0^1\bigl(
\sigma_p^2\sigma_q+\sigma_p
\sigma_q^2\bigr) (t) \,dt.
\]
An interesting example is the case $d=2$ with spot volatilities $\sigma
_1^2(t)=\sigma_2^2(t)=\sigma^2(t)$ and general correlation $\rho(t)$,
that is, $\sigma_{12}(t)=(\rho\sigma_1\sigma_2)(t)$. In this case,
we obtain
\begin{eqnarray*}
\AVAR \biggl(\int_0^1\sigma_{1}^2(t)
\,dt \biggr)&=&4\eta\int_0^1 \sigma
^3(t) \bigl(\sqrt{1+\rho(t)}+\sqrt{1-\rho(t)} \bigr) \,dt,
\\
\AVAR \biggl(\int_0^1\sigma_{12}(t)
\,dt \biggr) &=&2\eta\int_0^1 \sigma
^3(t) \bigl(\bigl(1+\rho(t)\bigr)^{3/2}+\bigl(1-\rho(t)
\bigr)^{3/2}\bigr) \,dt.
\end{eqnarray*}
With time-constant parameters, these bounds decay for $\sigma_1^2$
(resp., grow for $\sigma_{12}$) in $\vert\rho \vert$ from $8\eta
\sigma^3$
(resp., $4\eta\sigma^3$) at $\rho=0$ to $4\sqrt{2}\eta\sigma^3$
at $\vert\rho \vert=1$ for both cases.

\begin{figure}

\includegraphics{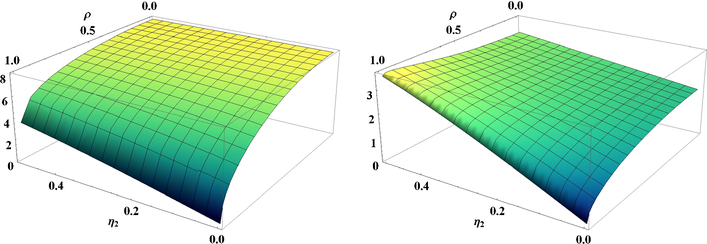}

\caption{Asymptotic variances of LMM for volatility
$\sigma_1^2$ (left) and co-volatility $\sigma_{12}$ (right)
plotted against correlation $\rho$ and noise level
$\eta_2$ (constant in time).}\label{fig1}
\end{figure}

Figure~\ref{fig1} illustrates the asymptotic variance in the case of
volatilities $\sigma^2_1=\sigma^2_2=1$ and co-volatility $\sigma
_{12}=\rho$ (constant in time) and the first noise level given by
$\eta
_1=1$. The left plot shows the asymptotic variance of the estimator of
$\sigma_1^2$ as a function of $\rho$ and $\eta_2$. It is shown that
using observations from the other (correlated) process induces clear
efficiency gains rising in $\rho$. If the noise level $\eta_2$ for the
second process is small, the asymptotic variance can even approach
zero. The plot on the right shows the same dependence for estimating
the co-volatility $\sigma_{12}$. For comparable size of $\eta_2$ and
$\eta_1$ the asymptotic variance increases in $\rho$, which is
explained by the fact that also the value to be estimated increases.
For small values of $\eta_2$, however, the efficiency gain by
exploiting the correlation prevails.

For larger dimensions $d$, the variance can even be of order $\mathcal
{O}(1/\sqrt{d})$: in the concrete case where all volatilities and noise
levels equal $1$, the asymptotic variance for estimating $\sigma_1^2$
can be reduced from $8$ (using only observations from the first
component or if $\Sigma$ is diagonal) down to $8/\sqrt{d}$ (in case of
perfect correlation).

All the preceding examples can be worked out for different noise levels
$\eta_p$. For a fixed entry $(p,q)$, generally all noise levels enter
and can be only de-coupled in case of a diagonal covariation matrix
$\Sigma=\diag(\sigma_p^2)_{1\le p\le d}$. Then the covariance
simplifies to
\[
p\neq q\dvtx\qquad 2\int_0^1 \bigl(
\eta_p\sigma_p\sigma_q^2+
\eta _q\sigma _q\sigma _p^2
\bigr) (t) \,dt;\qquad p=q\dvtx\qquad 8\int_0^1 \bigl(
\eta_p\sigma_p^3 \bigr) (t) \,dt.
\]

Finally, we can also investigate the estimation of the entire quadratic
covariation matrix $\int_0^1\Sigma(t) \,dt$ under homogeneous noise level
and measure its loss by the squared ($d\times d$)-Hilbert--Schmidt
norm. Summing up the variances for each entry, we obtain the asymptotic risk
\[
\frac{4\eta}{\sqrt n}\int_0^1 \bigl(\trace
\bigl(\Sigma^{1/2}\bigr)\trace (\Sigma )+\trace\bigl(
\Sigma^{3/2}\bigr) \bigr) (t) \,dt.
\]
This can be compared with the corresponding Hilbert--Schmidt norm error
$\frac{1}n(\trace(\Sigma)^2+\trace(\Sigma^2))$\vspace*{1pt} for the empirical
covariance matrix in the i.i.d. Gaussian $\mathbf{N}(0,\Sigma)$-setting.


\section{From discrete to continuous-time observations}
\label{sec:2}
\subsection{Setting}

First, let us specify different regularity assumptions. For functions
$f\dvtx[0,1]\to\R^m$, $m\ge1$ or also $m=d\times d$ for matrix
values, we
introduce the $L^2$-Sobolev ball of order $\alpha\in(0,1]$ and radius
$R>0$ given by
\begin{eqnarray}
H^{\alpha}(R)=\bigl\{f\in H^{\alpha}\bigl([0,1],\mathds
{R}^{m}\bigr)|\|f\|_{H^{\alpha}}\le R\bigr\}\nonumber\\
\eqntext{\mbox{where }\displaystyle\|f
\|_{H^{\alpha}}\: =\max_{1\le i\le m}\|f_{i}
\|_{H^{\alpha}},}
\end{eqnarray}
which for matrices means $\|f\|_{H^{\alpha}}\:=\max_{1\le i,j\le d}\|
f_{ij}\|_{H^{\alpha}}$.
We also consider H\"older spaces $C^\alpha([0,1])$ and Besov spaces
$B^{\alpha}_{p,q}([0,1])$ of such functions.
Canonically, for matrices we use the spectral norm $\|\cdot\|$ and we
set $\|f\|_{\infty}\:=\sup_{t\in[0,1]}{\|f(t)\|}$.

In order to pursue asymptotic theory, we impose that the deterministic
samplings in each component can be transferred to an equidistant scheme
by respective quantile transformations independent of $n_l,1\le l\le d$.

\begin{assump}[\hspace*{-4pt}($\alpha$)]\label{sampling} Suppose that
there exist
differentiable distribution functions $F_l$ with $F_l'\in C^\alpha
([0,1])$, $F_l(0)=0$, $F_l(1)=1$ and $F_l'>0$ such that the observation
times in \eqref{E0} are generated by $t_i^{(l)}=F_l^{-1}(i/n_l)$,
$0\le
i\le n_l$, $1\le l\le d$.
\end{assump}

We gather all assertions on the instantaneous co-volatility matrix
function $\Sigma(t)$, $t\in[0,1]$, which we shall require at some point.

\begin{assump}\label{smoothness}
Let $\Sigma\dvtx[0,1]\to\R^{d\times d}$ be a possibly random
function with
values in the class of symmetric, positive semi-definite matrices,
independent of $X$ and the observational noise, satisfying:
\begin{longlist}[(iii-$\underline\Sigma$)]
\item[(i-$\beta$)]$\Sigma\in H^{\beta}([0,1])$ for $\beta>0$.
\item[(ii-$\alpha$)]$\Sigma=\Sigma^B+\Sigma^M$ with $\Sigma^B\in
B_{1,\infty}^{\alpha}([0,1])$ for $\alpha>0$ and $\Sigma^M$ a
matrix-valued $L^2$-martingale.
\item[(iii-$\underline\Sigma$)] $\Sigma(t)\ge\underline{\Sigma}$
for a
strictly positive definite matrix $\underline\Sigma$ and all $t\in[0,1]$.
\end{longlist}
\end{assump}

We briefly discuss the different function spaces; see, for example,
Cohen \cite{Cohen}, Section~3.2, for a survey. First, any $\alpha$-H\"
older-continuous function lies in the $L^2$-Sobolev space $H^\alpha$
and any $H^\alpha$-function lies in the Besov space $B^\alpha
_{1,\infty
}$, where differentiability is measured in an $L^1$-sense. The
important class of bounded variation functions (e.g., modeling jumps in
the volatility) lies in $B^1_{1,\infty}$, but only in $H^\alpha$ for
$\alpha<1/2$. In particular, part (ii-$\alpha$), $\alpha\le1$, covers
$L^2$-semi-martingales by separate bounds on the drift (bounded
variation) and martingale part. Beyond classical theory in this area is
the fact that also nonsemi-martingales like fractional Brownian motion
$B^H$ with hurst parameter $H>1/2$ give rise to feasible volatility
functions in the results below, using $B^H\in C^{H-\eps}\cap
B^H_{1,\infty}$ for any $\eps>0$ as in Ciesielski {et al.}~\cite
{Ciesielski}.

In the sequel, the potential randomness of $\Sigma$ is often not
discussed additionally because by independence we can always work
conditionally on $\Sigma$. Finally, let us point out that we could
weaken the H\"older-assumptions on $F_1,\ldots,F_d$ toward Sobolev or
Besov regularity at the cost of tightening the assumptions on $\Sigma$.
For the sake of clarity, this is not pursued here.

Throughout the article, we write $Z_n=\mathcal{O}_P(\delta_n)$ and
$Z_n=\KLEINO_P(\delta_n)$ for a sequence of random variables $Z_n$ and
a sequence $\delta_n$, to express that $\delta_n^{-1}Z_n$ is tight and
tends to zero in probability, respectively. Analogously, $\mathcal{O}$
(or equivalently $\lesssim$) and $\KLEINO$ refer to deterministic
sequences. We write $Z_n\asymp Y_n$ if $Z_n=\mathcal{O}_P(Y_n)$ and
$Y_n=\mathcal{O}_P(Z_n)$ and the same for deterministic quantities.

\subsection{Continuous-time experiment}

\begin{defi}
Let $\mathcal{E}_0((n_l)_{1\le l\le d},\beta,R)$ with $n_l\in\mathds
{N},\beta\in(0,1],R>0$, be the statistical experiment generated by
observations from \eqref{E0} with $\Sigma\in H^{\beta}(R)$.
Analogously, let $\mathcal{E}_1((n_l)_{1\le l\le d},\beta,R)$ be the
statistical experiment generated by observing \eqref{E1} with the same
parameter class.
\end{defi}
As we shall establish next, experiments \eqref{E0} and \eqref{E1} will
be asymptotically equivalent as $n_l\to\infty,1\le l\le d$, at a
comparable speed, denoting
\[
n_{\mathrm{min}}=\min_{1\le l\le d}{n_l}
\quad\mbox{and}\quad n_{\mathrm{max}}=\max_{1\le l\le
d}{n_l}.
\]

\begin{thmm}\label{theo1}
Grant Assumption~\ref{sampling} with $\alpha=\beta$ on the design. The
statistical experiments $\mathcal{E}_0((n_l)_{1\le l\le d},\beta,R)$ and
$\mathcal{E}_1((n_l)_{1\le l\le d},\beta,R)$ are asymptotically
equivalent for any $\beta\in(0,1/2]$ and $R>0$, provided
$n_{\mathrm{min}}\rightarrow\infty$, $n_{\mathrm{max}}=\KLEINO((n_{\mathrm{min}})^{1+\beta})$.
More precisely, the Le Cam distance $\Delta$ is of order
\[
\Delta \bigl(\mathcal{E}_0\bigl((n_l)_{1\le l\le d},
\beta,R\bigr),\mathcal{E}_1\bigl((n_l)_{1\le l\le d},
\beta,R\bigr) \bigr)= \mathcal{O} \Biggl(R^2 \Biggl(\sum
_{l=1}^dn_l/\eta_l^2
\Biggr)n_{\mathrm{min}}^{-1-\beta
} \Biggr). 
\]
\end{thmm}

By inclusion, the result also applies for $\beta>1/2$ when in the
remaining expressions $\beta$ is replaced by $\min(\beta,1/2)$.
A standard Sobolev smoothness of $\Sigma$ is $\beta$ almost 1/2 for
diffusions with finitely many or absolutely summable jumps. In that
case, the asymptotic equivalence result holds if $n_{\mathrm{max}}$ grows more
slowly than~$n_{\mathrm{min}}^{3/2}$.
Theorem~\ref{theo1} is proved in the \hyperref[app]{Appendix} in a
constructive way by
warped linear interpolation, which yields a readily implementable
procedure; cf. Section~\ref{sec:5} below.


\section{Localisation and method of moments}
\label{sec:3}

\subsection{Construction}

We partition the interval $[0,1]$ in blocks $[kh,(k+1)h)$ of length
$h$. On each block a parametric MLE for a constant model could be
sought for. Its numerical determination, however, is difficult and
unstable due to the nonconcavity of the ML objective function and its
analysis is quite involved. Yet, the likelihood equation leads to
spectral statistics whose empirical covariances estimate the quadratic
covariation matrix. We therefore prefer a localised method of moments
(LMM) for these spectral statistics where for an adaptive version the
theoretically optimal weights are determined in a pre-estimation step,
in analogy with the classical (multi-step) GMM (generalised method of
moments) approach by Hansen \cite{Hansen}.

As motivated in Section~\ref{sec:1b}, let us consider
the local spectral statistics
$S_{jk}$ in \eqref{spec} from the continuous-time experiment \eqref{E1}.
First, we consider a locally constant approximation.

\begin{defi}
Set $\bar{f}_h(t):=h^{-1}\int_{kh}^{(k+1)h}f(s)\,ds$ for $t\in
[kh,(k+1)h)$ and a function $f$ on $[0,1]$. Assume $h^{-1}\in\mathds
{N}$ and let $X_t^h=X_0+\int_0^t\bar\Sigma^{{1}/{2}}_{h}(s)
\,d{B}_s$ with a $d$-dimensional standard Brownian motion ${B}$. Define
the process
%
{\renewcommand{\theequation}{$\mathcal{E}_2$}
\begin{equation}
\label{E2}
d\tilde{Y}_t=X_t^h
\,dt+\diag \Bigl(\sqrt{\overline {H}^2_{n,l,h}(t)}
\Bigr)_{1\le l\le d } \,d{W}_t,\qquad t\in[0,1],
\end{equation}}
\hspace*{-6pt}where ${W}$ is a standard Brownian motion independent of ${B}$ and with
noise level \eqref{noiselevel}. The observations from \eqref{E2} for
$\Sigma\in H^{\beta}(R)$ generate experiment $\mathcal
{E}_2((n_l)_{1\le
l\le d},h,\beta,R)$.
\end{defi}

In experiment \eqref{E2}, we thus observe a process with a
co-volatility matrix which is constant on each block $[kh,(k+1)h)$ and
corrupted by noise of block-wise constant magnitude. Our approach is
founded on the idea that for small block sizes $h$ and sufficient
regularity this piecewise constant approximation is close to \eqref{E1}.


The LMM estimator is built from the data in experiment $\mathcal{E}_1$,
but designed for the block-wise parametric model \eqref{E2}. In \eqref
{E2}, the $L^2$-orthogonality of $(\phi_{jk})$ as well as that of
$(\Phi
_{jk})$ imply (cf. Rei{ss} \cite{reiss})
%
\setcounter{equation}{0}
\begin{equation}
\label{lawSjk} S_{jk}\sim\mathbf{N} ({0},C_{jk} )\qquad
\mbox{independent for all }  (j,k)
\end{equation}
with covariance matrix
%
\begin{eqnarray}
\label{localnorm} C_{jk}&=& \Sigma^{kh}+\pi^2j^2h^{-2}
\diag\bigl(H^{kh}_{n,l}\bigr)_l^2,\qquad
\Sigma ^{kh}= \bar{\Sigma}_h(kh),
\nonumber
\\[-8pt]
\\[-8pt]
\nonumber
 H^{kh}_{n,l}&=&
\bigl(\overline {H}^2_{n,l,h}(kh)\bigr)^{1/2}.
\end{eqnarray}
Let us further introduce the Fisher information-type matrices
\[
I_{jk}=C_{jk}^{-1}\otimes C_{jk}^{-1},\qquad
I_k=\sum_{j=1}^{\infty}I_{jk},\qquad
j\ge1, k=0,\ldots,h^{-1}-1.
\]

Our local method of moments estimator with oracle weights
$\operatorname
{LMM}_{\mathrm{or}}^{(n)}$ exploits that on each block a natural second moment
estimator of $\Sigma^{kh}$ is given as a convex combination of the
bias-corrected empirical covariances:
%
\begin{equation}\qquad
\label{orlle} \operatorname{LMM}_{\mathrm{or}}^{(n)}:=\sum
_{k=0}^{h^{-1}-1} h \sum_{j=1}^{\infty}W_{jk}
\operatorname{vec} \biggl(S_{jk}S_{jk}^{\top} -\frac{\pi^2j^2}{h^{2}}
\diag{\bigl(\bigl(H_{n,l}^{kh}\bigr)^2
\bigr)}_{1\le l\le
d} \biggr).
\end{equation}
The optimal weight matrices $W_{jk}$ in the oracle case are obtained as
%
\begin{equation}
\label{EqWjk} W_{jk}:=I_k^{-1}I_{jk}
\in\mathds{R}^{d^2\times d^2}.
\end{equation}
Note that $C_{jk}, I_{jk},I_k$ and $W_{jk}$ all depend on $(n_l)_{1\le
l\le d}$ and $h$, which is omitted in the notation. Finally, observe
that \eqref{localnorm} and $\sum_jW_{jk}=E_{d^2}$ imply that
$\operatorname{LMM}_{\mathrm{or}}^{(n)}$ is unbiased under model \eqref{E2}.

\subsection{Asymptotic properties of the estimators}

We formulate the main result of this section that the oracle estimator
\eqref{orlle} and also a fully adaptive version for the quadratic
covariation matrix satisfy central limit theorems.

\begin{thmm}\label{cltorlle}
Let Assumptions \ref{sampling}($\alpha$), \ref
{smoothness}(\textup{ii}-$\alpha
$) and \ref{smoothness}(\textup{iii}-$\underline\Sigma$) with $\alpha>1/2$ hold
true for observations from model \eqref{E1}. The oracle estimator
\eqref
{orlle} yields a consistent estimator for $\operatorname{vec}(\int_0^1\Sigma(s) \,ds)$
as $n_{\mathrm{min}}\rightarrow\infty$ and $h=h_0n_{\mathrm{min}}^{-1/2}$ with $h_0\to
\infty$. Moreover, if $n_{\mathrm{max}}=\KLEINO(n_{\mathrm{min}}^{2\alpha})$ and
$h=\KLEINO(n_{\mathrm{max}}^{-1/4})$,
then a multi-variate central limit theorem holds:
%
\begin{equation}
\label{fclt} \mathbf{I}_n^{1/2} \biggl(
\operatorname{LMM}_{\mathrm{or}}^{(n)} -\operatorname{vec} \biggl(\int
_0^1\Sigma(s) \,ds \biggr) \biggr)\stackrel{
\mathcal {L}} {\longrightarrow} \mathbf{N}(0,{\cal Z})\qquad \mbox{in }
\mathcal{E}_1
\end{equation}
with $\cal Z$ from \eqref{EqZ} and $\mathbf{I}_n^{-1}=\sum_{k=0}^{h^{-1}-1}h^2I_{k}^{-1}$.
\end{thmm}

While the preceding result is most useful in applications, it is, of
course, important to understand the asymptotic covariance structure of
the estimator as well; cf. the discussion of efficiency above.
Therefore, we consider comparable sample sizes and normalise with
$n_{\mathrm{min}}^{1/4}$ in the following result.

\begin{cor} \label{corclt}
Under the assumptions of Theorem~\ref{cltorlle} suppose\break
$n_{\mathrm{min}}/  n_p\to
\nu_p\in(0,1]$ for $p=1,\ldots,d$ and introduce ${\cal H}(t)=\diag
(\eta
_p\nu_p^{1/2}\times F_p'(t)^{-1/2})_{p}\in\R^{d\times d}$ and $\Sigma
_{\cal
H}^{1/2}:={\cal H}({\cal H}^{-1}\Sigma{\cal H}^{-1})^{1/2}{\cal H}$. Then
%
\begin{equation}
\label{fclt2} n_{\mathrm{min}}^{1/4} \biggl(\operatorname{LMM}_{\mathrm{or}}^{(n)}
-\operatorname{vec} \biggl(\int_0^1\Sigma(s) \,ds \biggr)
\biggr)\stackrel{\mathcal {L}} {\longrightarrow} \mathbf{N} \bigl(0,
\mathbf{I}^{-1}{\cal Z} \bigr)\qquad \mbox{in } \mathcal{E}_1
\end{equation}
with $\mathbf{I}^{-1}=2
\int_0^1(\Sigma\otimes\Sigma_{\cal H}^{1/2} + \Sigma_{\cal H}^{1/2}
\otimes\Sigma)(t) \,dt$.
In particular, the entries satisfy for $p,q=1,\ldots,d$
%
\begin{eqnarray}
\label{avar}&& n_{\mathrm{min}}^{1/4} \biggl(\bigl(
\operatorname{LMM}_{\mathrm{or}}^{(n)}\bigr)_{p(d-1)+q} -\int
_0^1\Sigma_{pq}(s) \,ds \biggr)
\nonumber
\\
&&\qquad\stackrel{\mathcal {L}} {\longrightarrow}\mathbf{N}
\biggl(0,2(1+\delta_{p,q})\\
&&\hspace*{57pt}{}\times\int
_0^1\bigl(\Sigma _{pp}\bigl(\Sigma
_{\cal H}^{1/2}\bigr)_{qq}+\Sigma_{qq}
\bigl(\Sigma_{\cal H}^{1/2}\bigr)_{pp}+2\Sigma
_{pq}\bigl(\Sigma_{\cal H}^{1/2}\bigr)_{pq}
\bigr) (t) \,dt \biggr). \nonumber
\end{eqnarray}
\end{cor}

The variance \eqref{avar} will coincide with the lower bound obtained
in Section~\ref{sec:4} below. The local noise level in ${\cal H}(t)$
depends on the observational noise level $\eta_p$ and the local sample
size $\nu_p^{-1}F_p'(t)$, $p=1,\ldots,d$, after normalisation by $n_{\mathrm{min}}$.
It is easy to see that in the case $n_{\mathrm{min}}/n_p\to0$ the asymptotic
variance vanishes for all entries $(p,q)$, $q=1,\ldots,d$. We infer the
structure of the asymptotic covariance matrix using block-wise
diagonalisation in Appendix~\ref{appb}.

To obtain a feasible estimator, the optimal weight matrices
$W_{jk}=W_j(\Sigma^{kh})$ and the information-type matrices
$I_{jk}=I_j(\Sigma^{kh})$ are estimated in a preliminary step from the
same data. To reduce variability in the estimate, a coarser grid of
$r^{-1}$ equidistant intervals, $r/h\in\mathds{N}$ is employed for
$\hat W_{jk}$. As derived in Bibinger and Rei{ss} \cite{bibingerreiss}
for supremum norm loss and extended to $L^1$-loss and Besov regularity
using the $L^1$-modulus of continuity as in the case of wavelet
estimators (Corollary~3.3.1 in Cohen \cite{Cohen}), a preliminary
estimator $\hat\Sigma(t)$ of the instantaneous co-volatility matrix
$\Sigma(t)$ exists with
%
\begin{equation}
\label{EqPilot} \Vert\hat\Sigma-\Sigma \Vert _{L^1}=\mathcal
{O}_P \bigl(n_{\mathrm{min}}^{-\alpha/(4\alpha+2)} \bigr)
\end{equation}
for $\Sigma\in B^\alpha_{1,\infty}([0,1])$. For block $k$ with
$kh\in
[mr,(m+1)r)$, we set
\[
\hat W_{jk}=W_j\bigl(\hat\Sigma^{mr}\bigr),\qquad
\hat I_{jk}=I_j\bigl(\hat\Sigma ^{kh}\bigr)\qquad
\mbox {with } \hat\Sigma^{mr}=\overline{\hat\Sigma}_r(mr),
\hat\Sigma ^{kh}=\overline{\hat\Sigma}_h(kh).
\]
The LMM estimator with adaptive weights is then given by
%
\begin{equation}\qquad
\label{adlle} \operatorname{LMM}_{\mathrm{ad}}^{(n)} =\sum
_{k=0}^{h^{-1}-1}h \sum_{j=1}^{\infty}
\hat W_{jk} \operatorname{vec} \biggl(S_{jk}S_{jk}^{\top}
-\frac{\pi^2j^2}{h^{2}}\diag{\bigl(\bigl(H_{n,l}^{kh}
\bigr)^2\bigr)}_{1\le l\le
d} \biggr).
\end{equation}
We estimate the total covariance matrix via
%
\begin{equation}
\label{EqIhat} \hat{\mathbf{I}}_n^{-1}=\sum
_{k=0}^{h^{-1}-1}h^2 \Biggl(\sum
_{j=1}^\infty \hat I_{jk}
\Biggr)^{-1}.
\end{equation}
As $j\to\infty$, the weights $W_j(\Sigma)$ and the matrices
$I_j(\Sigma
)$ decay like $j^{-4}$ in norm, compare Lemma~\ref{LemWjk} below, such
that in practice a finite sum over frequencies $j$ suffices. By a tight
bound on the derivatives of $\Sigma\mapsto W_j(\Sigma)$, we show in
Appendix~\ref{App:Adaptive} the following general result.

\begin{thmm}\label{cltadlle}
Suppose $\Sigma\in B^\alpha_{1,\infty}([0,1])$ for $\alpha\in
(1/2,1]$ satisfying\break
$ \alpha/(2\alpha+1)>\log(n_{\mathrm{max}})/\log(n_{\mathrm{min}})-1.$
Choose $h,r\to0$ such that $h_0=\break hn_{\mathrm{min}}^{1/2}\asymp\log(n_{\mathrm{min}})$
and $n_{\mathrm{min}}^{-\alpha/(2\alpha+1)}\lesssim r\lesssim
(n_{\mathrm{min}}/n_{\mathrm{max}})^{1/2}$, $h^{-1},r^{-1},r/h\in\N$. If the pilot
estimator $\hat\Sigma$ satisfies \eqref{EqPilot}, then under the
conditions of Theorem~\ref{cltorlle} the adaptive estimator~\eqref
{adlle} satisfies
%
\begin{equation}
\label{fcltad} \hat{\mathbf{I}}_n^{1/2} \biggl(
\operatorname{LMM}_{\mathrm{ad}}^{(n)} -\operatorname{vec} \biggl(\int
_0^1\Sigma(s) \,ds \biggr) \biggr)\stackrel{
\mathcal {L}} {\longrightarrow} \mathbf{N} (0,{\cal Z} ),
\end{equation}
with $\hat{\mathbf{I}}_n$ from \eqref{EqIhat}.

Moreover, Corollary~\ref{corclt} applies equally to the adaptive
estimator \eqref{adlle}.
\end{thmm}

Since the estimated $\hat{\mathbf I}_n$ appears in the CLT, we have
obtained a feasible limit theorem and (asymptotic) inference statements
are immediate.

Some assumptions of Theorem~\ref{cltadlle} are tighter than for the
oracle estimator. To some extent this is for the sake of clarity. Here,
we have restricted Assumption~\ref{smoothness}(ii-$\alpha$) to the
Besov-regular part. A generalisation of the pilot estimator to
martingales seems feasible, but is nonstandard and might require
additional conditions. We have also proposed a concrete order of $h$
and $r$, less restrictive bounds are used in the proof; see, for
example, \eqref{Eqdelta_n} below.

The lower bound for $\alpha$ in terms of the sample-size ratio
$n_{\mathrm{max}}/n_{\mathrm{min}}$ is due to rough norm bounds for (estimated)
information-type matrices. For $\alpha=1$ (bounded variation case), the
restriction imposes $n_{\mathrm{max}}$ to be slightly smaller than
$n_{\mathrm{min}}^{4/3}$. By the Sobolev embedding $B^1_{1,\infty}\subset
H^\beta
$ for all $\beta<1/2$, the restriction $n_{\mathrm{max}}=\KLEINO
(n_{\mathrm{min}}^{1+\beta
})$ from Theorem~\ref{theo1} is clearly also satisfied in this case.
It is not clear whether a more elaborate analysis can avoid these
restrictions. Still, to the best of our knowledge, a feasible CLT for
asymptotically separating sample sizes has not been obtained before.


\section{Semi-parametric Cram\'er--Rao bound}\label{sec:4}

We shall derive an efficiency bound for the following basic case of
observation model \eqref{E1}:
%
\begin{equation}
\label{CRmodelY} dY_t=X_t \,dt+\frac{1}{\sqrt{n}}
\,dW_t, \qquad X_t=\int_0^t
\Sigma(s)^{1/2}\,dB_s,\qquad  t\in[0,1],
\end{equation}
where
%
\begin{equation}
\label{CRmodelX} \Sigma(t)=\Sigma_{0}(t)+\eps\mathds{H}(t), \qquad\Sigma
_{0}(t)^{1/2}=O(t)^\top\Lambda(t)O(t).
\end{equation}
We assume $\Sigma_0(t)$ and $\mathds{H}(t)$ to be known symmetric
matrices, $O(t)$ orthogonal matrices, $\Lambda(t)=\diag(\lambda
_1(t),\ldots,\lambda_d(t))$ diagonal and consider $\eps\in[-1,1]$ as
unknown parameter. Furthermore, we require Assumption~\ref
{smoothness}(iii-$\underline\Sigma$) for all $\Sigma$. Finally, we
impose throughout this section the regularity assumption that the
matrix functions $O(t),\mathds{H}(t),\Lambda(t)$ are continuously
differentiable.

The key idea is to transform the observation of $dY_t$ in such a manner
that the white noise part remains invariant in law while for the
central parameter $\Sigma(t)=\Sigma_0(t)$ the process $X$ is
transformed to a process with independent coordinates and constant
volatility. It turns out that this can only be achieved at the cost of
an additional drift in the signal. The construction first rotates the
observations via $O(t)$, which diagonalises $\Sigma_0(t)$, and then
applies a coordinate-wise time-transformation, corrected by a
multiplication term to ensure $L^2$-isometry such that the white noise
remains law-invariant. All proofs are delegated to the supplementary
file \cite{supplement}.

We introduce the coordinate-wise time changes by
\[
r_i(t)=\frac{\int_0^t\lambda_i(s)\,ds} {\int_0^1\lambda_i(s)\,ds}
\quad\mbox{and} \quad(T_rg) (t):=
\bigl(g_1\bigl(r_1(t)\bigr),\ldots,g_d
\bigl(r_d(t)\bigr)\bigr)^\top
\]
for $g=(g_1,\ldots,g_d)\dvtx\R\to\R^d$. Moreover, we set
\[
\bar\Lambda:=\int_0^1 \Lambda(s)\,ds,\qquad
R'(t):=\bar\Lambda ^{-1}\Lambda (t)=\diag
\bigl(r_1'(t),\ldots,r_d'(t)
\bigr).
\]

\begin{lem}\label{LemTrafo}
By transforming $d\bar Y=T_r^{-1}{\cal M}_{(R')^{-1/2}O}\,dY$, the
observation model \eqref{CRmodelY}, \eqref{CRmodelX} is equivalent to
observing
%
\begin{equation}
\label{EqTrafo}d\bar Y(t)=S(t) \,dt+\frac{1}{\sqrt{n}}\,d\bar W(t)
\end{equation}
with
\begin{eqnarray*}
S(t)&=&T_r^{-1} \biggl(\bigl(R'
\bigr)^{-1} \biggl(\int_0^\cdot\bigl(
\bigl(R'\bigr)^{-1/2}O\bigr)'(s)X(s) \,ds \\
&&\hspace*{61pt}{}+ \int
_0^\cdot\bigl(R'(s)
\bigr)^{-1/2}O(s) \,dX(s) \biggr) \biggr) (t)
\end{eqnarray*}
for $t\in[0,1]$. At $\eps=0$ the observation $d\bar Y(t)$ reduces to
%
\begin{equation}\quad
\label{EqCentral} \biggl(\int_0^tT_r^{-1}
\bigl(\bigl(R'\bigr)^{-1}\bigl(\bigl(R'
\bigr)^{-1/2}O\bigr)'X\bigr) (s) \,ds+\bar\Lambda \bar B(t)
\biggr) \,dt +\frac{1}{\sqrt{n}}\,d\bar W(t).
\end{equation}
Here $\bar W$ and $\bar B$ are Brownian motions
obtained from $W$ and $B$, respectively, via rotation and time shift.
\end{lem}

If we may forget in \eqref{EqCentral} the first term, which is a drift
term with respect to the martingale part $\bar\Lambda\bar B(t)$, then
the central observation is indeed a constant volatility model in white noise.

Let us introduce the multiplication operator ${\cal M}_Ag:=Ag$ and
the integration operator $Ig(t)=-\int_t^1g(s) \,ds$ and its adjoint
$I^\ast g(t)=-\int_0^tg(s) \,ds$.
The covariance operator $C_{n,\eps}$ on $L^2([0,1],\R^d)$ obtained from
observing the differential in \eqref{EqTrafo} is then given by
\[
C_{n,\eps}=T_r^\ast{\cal M}_{(R')^{1/2}O}
I^\ast{\cal M}_{\Sigma
_0+\eps\mathds{H}} I {\cal M}_{O^\top(R')^{1/2}}T_r+n^{-1}
\Id.
\]
The covariance operator $Q_{n,\eps}$ when omitting the drift part is
given by
\[
Q_{n,\eps}=Q_{n,0}+\eps I^\ast T_r^\ast{
\cal M}_{M} T_r I\qquad\mbox{with } M(t):=\bigl(
\bigl(R'\bigr)^{-1/2}O\mathds{H}O^\top
\bigl(R'\bigr)^{-1/2}\bigr) (t),
\]
where for $\eps=0$ the one-dimensional Brownian motion covariance
operator $C_{\mathrm{BM}}=I^\ast I$ appears in
$Q_{n,0}=\diag(\bar\lambda_{ii}C_{\mathrm{BM}}+n^{-1}\Id)_{1\le i\le d}$.

Standard calculations for the finite-dimensional Gaussian scale model,
for example, \cite{LehmannCasella}, Chapter~6.6, transfer one-to-one
to the
infinite-dimensional case of observing $\mathbf{N}(0,Q_{n,\eps})$ and
yield as Fisher information for the parameter $\eps$ at $\eps=0$ the value
$I_n^Q=\frac{1}2 \Vert Q_{n,0}^{-1/2}\dot Q_{0}Q_{n,0}^{-1/2} \Vert_{\mathrm{HS}}^2 $
because $Q_{n,0}^{-1/2}Q_{n,\eps}Q_{n,0}^{-1/2}$ is differentiable at
$\eps=0$ in Hilbert--Schmidt norm. We show by Hilbert--Schmidt
calculus, the Feldman--Hajek theorem and the Girsanov theorem that the
models with and without drift do not separate:
%
\begin{equation}
\label{LemmaNoDrift} \limsup_{n\to\infty}\bigl\Vert Q_{n,0}^{-1/2}
\dot Q_0 Q_{n,0}^{-1/2} - C_{n,0}^{-1/2}
\dot C_0 C_{n,0}^{-1/2} \bigr\Vert_{\mathrm{HS}}<
\infty.
\end{equation}
Consequently, the drift only contributes the negligible order $\mathcal
{O}(1)=\KLEINO(\sqrt{n})$ to the Fisher information.
Analysing $\mathbf{N}(0,Q_{n,\eps})$, we thus establish a
semi-parametric Cram\'er--Rao bound for estimating any linear
functional of the co-volatility matrix.

\begin{thmm}\label{ThmCR}
For a continuous matrix-valued function $A\dvtx[0,1]\to\R^{d\times d}$
consider the estimation of
%
\begin{equation}
\label{EqFunctional} \theta:=\int_0^1\bigl\langle
A(t),\Sigma(t) \bigr\rangle_{\mathrm{HS}} \,dt =\int_0^1
\sum_{i,j=1}^dA_{ij}(t)
\Sigma_{ij}(t) \,dt\in\R.
\end{equation}
Then a hardest parametric subproblem in model \eqref{CRmodelY}, \eqref
{CRmodelX} is obtained for the perturbation of $\Sigma_0$ by
\[
\mathds{H}^\ast(t)=\bigl(\Sigma_0 \bigl(A+A^\top
\bigr)\Sigma_0^{1/2}+\Sigma_0^{1/2}
\bigl(A+A^\top\bigr)\Sigma_0\bigr) (t).
\]
There any estimator $\hat\theta_n$ of $\theta$, which is asymptotically
unbiased in the sense $\frac{d}{d\theta}(\E_\theta[\hat\theta
_n]-\theta
)\to0$, satisfies as $n\to\infty$
\begin{eqnarray*}
&&\Var_{\eps=0}(\hat\theta_n)\\
&&\qquad\ge\frac{(2+\KLEINO(1))}{\sqrt
{n}}\int
_0^1 \bigl\langle\bigl(\Sigma_0
\otimes\Sigma_0^{1/2}+\Sigma_0^{1/2}
\otimes \Sigma _0\bigr){\cal Z}\operatorname{vec}(A),{\cal Z}\operatorname{vec}(A) \bigr
\rangle(t)\,dt.
\end{eqnarray*}
\end{thmm}

Further classical efficiency statements like the local asymptotic
minimax theorem would require the LAN-property of the parametric subproblem.

\section{Implementation and numerical results}
\label{sec:5}

\subsection{Discrete-time estimator}
The construction to transfer discrete-time to continuous-time
observations in the proof of Theorem~\ref{theo1} paves the way to the
discrete approximation of the local spectral statistics \eqref{spec}.
Using the interpolated process and integration by parts yields
\[
\int\phi_{jk}(t)\,dY^{(l)}(t)\asymp-\sum
_{\nu=1}^{n_l} \int_{t_{\nu
-1}^{(l)}}^{t_{\nu}^{(l)}}
\Phi_{jk}(t)\frac{Y_{\nu}^{(l)}-Y_{\nu
-1}^{(l)}}{t_{\nu}^{(l)}-t_{\nu-1}^{(l)}} \,dt.
\]
Hence, for discrete-time observations from \eqref{E0} we use the local
spectral statistics $S_{jk}$ in \eqref{specdis}.
The noise terms in \eqref{localnorm} translate from $\mathcal{E}_1$ to
$\mathcal{E}_0$ via substituting $n_l^{-1}\int_{kh}^{(k+1)h}(F_l^{\prime
}(s))^{-1} \,ds$ by $\sum_{\nu: kh\le t_{\nu}^{(l)}\le(k+1)h}(t_{\nu
}^{(l)}-t_{\nu-1}^{(l)})^2$. The discrete sum times $h^{-1}$ can be
understood as a block-wise quadratic variation of time in the spirit of
Zhang {et al.} \cite{zhangmykland}. The bias is discretised
analogously. In theory and practice, frequencies $j$ larger than $\log
(\eta_p^{-1}n^{1/2})$ can be cut off as the size of the weights $W_j$
decays rapidly for $j\to\infty$.
Different constants in the choice of the block size $h$ do not cause a
finite-sample bias, unless the volatility oscillates rapidly over time
(in a nonmartingale fashion).

For the adaptive estimator we are in need of local estimates of
$n_lF_l^{\prime}$, $\Sigma$ and estimators for $\eta_l^2,1\le l\le d$.
It is well known how to estimate noise variances with faster $\sqrt
{n_l}$-rates; see, for example, Zhang {et al.} \cite{zhangmykland}.
Local observation densities can be estimated with block-wise quadratic
variation of time
as above, which then yield estimates $\hat H_{n,l}^{kh}$ of $H_{n,l}$
around time $kh$. Uniformly consistent estimators for $\Sigma(t),t\in
[0,1]$, are feasible, for example, averaging spectral statistics for
$j=1,\ldots,J$ over a set ${\mathcal K}_t$ of $K$ adjacent blocks
containing $t$:
%
\begin{equation}
\label{Pilot} \hat{\Sigma }(t)=K^{-1}\sum
_{k\in{\mathcal K}_t} J^{-1}\sum_{j=1}^J
\bigl(S_{jk}S_{jk}^\top-\pi^2
j^2h^{-2}\diag\bigl(\bigl(\hat H_{n,l}^{kh}
\bigr)_l^2\bigr) \bigr).
\end{equation}
We refer to Bibinger and Rei{ss} \cite{bibingerreiss} for details on
the nonparametric pilot estimator with $J=1$.

\begin{figure}

\includegraphics{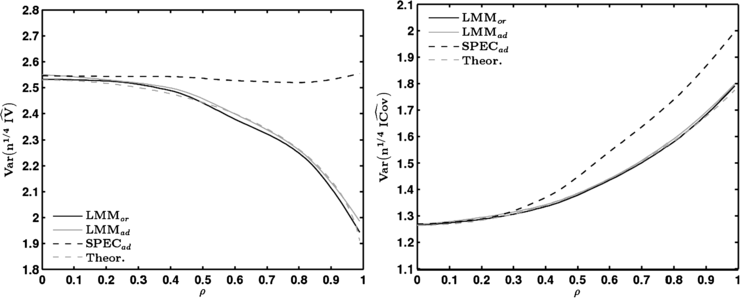}

\caption{Variances of estimators of
$\sigma_{1}^2$ (left) and $\sigma_{12}$ (right)
in time-constant scenario.}
\label{PAR_AVAR_V1}\vspace*{-3pt}
\end{figure}

\subsection{Simulations}
We examine the finite-sample properties of the LMM for the case $d=2$
in two scenarios. First, we compare the finite-sample variance with the
asymptotic variances from Sections \ref{sec:2} and \ref{sec:3}, for a parametric setup with
$\eta_1^2=\eta_2^2=0.1$, $\sigma_1=\sigma_2=1$ and constant correlation
$\rho$. We simulate $n_1=n_2=30\mbox{,}000$ synchronous observations on
$[0,1]$. For estimating $\sigma_{1}^2$ and $\sigma_{12}=\rho$,
Figure~\ref{PAR_AVAR_V1} displays the rescaled Monte-Carlo variance
based on
$20\mbox{,}000$ replications of the\vadjust{\goodbreak} oracle and adaptive LMM (LMM$_{\mathrm{or}}$ and
LMM$_{\mathrm{ad}}$), as well as the adaptive spectral estimator (SPEC$_{\mathrm{ad}}$)
by Bibinger and Rei{ss} \cite{bibingerreiss}. The latter relies on the
same spectral approach, but uses only scalar weighting instead of the
full information matrix approach.

In practice, the pilot estimator from \eqref{Pilot} for $J$ not too
large performed well. As configuration we use $h^{-1}=10$, $J=30$ and
$K=8$, which turned out to be an accurate choice,
but the estimators are reasonably robust to alternative input choices.
For the LMM of $\sigma_{1}^2$, we observe the variance reduction effect
associated with a growing signal correlation $\rho$, while the
simulation-based variances of both LMM$_{\mathrm{or}}$ and LMM$_{\mathrm{ad}}$ are close
to their theoretical asymptotic counterpart (Theor.). The results for
$\sigma_{12}$ underline the precision gains compared to SPEC$_{\mathrm{ad}}$
with univariate weights when $\rho$ increases.

Next, we consider a complex and realistic stochastic volatility setting
that relies on an extension of the widely-used Heston model as, for
example, employed by A\"{i}t-Sahalia {et al.} \cite{aitfanxiu2010},
accounting for both leverage effects and an intraday seasonality of
volatility. The signal process for $l=1,2$ evolves as\looseness=-1
\[
\label{sim_setting} dX^{(l)}_t=\phi_l (t )
\sigma_l (t ) \,d Z_t^{(l)},\qquad  d
\sigma^2_l (t )=\alpha_l \bigl(
\mu_l-\sigma^2_l (t ) \bigr)\,dt +
\psi_l \sigma_l (t ) \,d V_t^{(l)},
\]\looseness=0
where $Z_t^{(l)}$ and $V_t^{(l)}$ are standard Brownian motions with
$dZ_t^{(1)} \,dZ_t^{(2)}=\rho \,dt$ and $dZ_t^{(l)} \,dV_t^{(m)}=\delta
_{l,m} \gamma_l \,dt$. $\phi_l  (t )$ is a nonstochastic
seasonal factor with $\int_0^1 \phi_l^2  (t ) \,dt=1$.
The unit time interval can represent one trading day, for example, 6.5
hours or 23,400 seconds at NYSE.

We initialise the variance process $\sigma^2_l  (t )$ by
sampling from its stationary distribution $\Gamma
 (2
\alpha_l \mu_l/\psi_l^2,\psi_l^2/(2\alpha_l) )$ and vary the value
of the instantaneous signal correlation $\rho$, while setting $
(\mu
_l,\alpha_l,\psi_l,\gamma_l )= (1,6,0.3,-0.3 )$, $l=1,2$,
which under the stationary distribution, implies \mbox{$\E [\int_0^1 \phi_l^2  (t )\sigma^2_l  (t ) \,dt ]=1$}.
The seasonal factor $\phi_l
 (t )$ is specified in terms of intraday volatility functions
estimated for S\&P 500 equity data by the procedure in Andersen and
Bollerslev~\cite{andbol1997}. $\phi_1  (t )$ and $\phi_2
(t )$ are based on cross-sectional averages of the 50 most and 50\vadjust{\goodbreak}
least liquid stocks, respectively, which yields a pronounced L-shape in
both cases (see Figure~\ref{intraday_vol_rmse}). We add noise processes
that are i.i.d. $\mathbf{N}  (0,\eta_l^2 )$ and mutually
independent with
$\eta_l=0.1(\E [\int_0^1 \phi_l^4  (t )\sigma^4_l
(t ) \,dt ])^{1/4}$, computed under the stationary distribution
of $\sigma^2_l
 (t )$. Finally, asynchronicity effects are introduced by
drawing observation times $t_i^{(l)}$, $1\leq i\leq n_l$, $l=1,2$, from
two independent Poisson processes with intensities $\lambda_1=1$ and
$\lambda_2=2/3$ such that, on average, $n_1=23\mbox{,}400$ and $n_2=15\mbox{,}600$.

\begin{figure}

\includegraphics{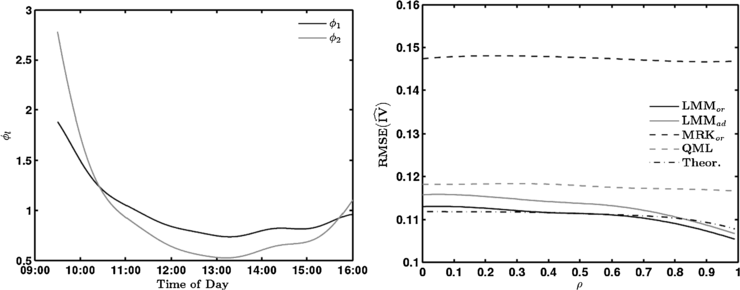}

\caption{Nonstochastic volatility
seasonality factors (left) and RMSE for estimators of
$\int_0^1 \phi_1^2  (t )\sigma^2_1  (t ) \,dt$
(right) in stochastic volatility scenario.}\label{intraday_vol_rmse}
\end{figure}

As a representative example, Figure~\ref{intraday_vol_rmse} \,depicts the
root mean-squared errors (RMSEs) based on $40\mbox{,}000$ replications of the
following estimators of $\int_0^1 \phi_1^2  (t )\sigma^2_1
 (t ) \,dt$: the oracle and adaptive LMM using $h^{-1}=20$,
$J=15$ and $K=8$, the quasi-maximum likelihood (QML) estimator by A\"
{i}t-Sahalia {et al.} \cite{aitfanxiu2010} as well as an oracle
version of the widely-used multivariate realised kernel (MRK$_{\mathrm{or}}$) by
Barndorff-Nielsen {et al.} \cite{{bn2011}}. For the latter, we
employ the average univariate mean-squared error optimal bandwidth
based on the true value of $\int_0^1 \phi_l^4  (t )\sigma
^4_l
 (t ) \,dt$, $l=1,2$. Finally, we include the theoretical
variance from the asymptotic theory (Theor.), which is computed as the
variance \eqref{avar} averaged across all replications.

Three major results emerge. First, the LMM offers considerable
precision gains when compared to both benchmarks. Second, a rising
instantaneous signal correlation $\rho$ is associated with a declining
RMSE of the LMM, which is due to the decreasing variance, and thus
confirms the findings from Section~\ref{sec:2} in a realistic setting. Finally,
the adaptive LMM closely tracks its oracle counterpart.

In summary, the simulation results show that the estimator has
promising properties even in settings which are more general than those
assumed in \eqref{E1}, allowing, for instance, for random observation
times, stochastic intraday volatility as well as leverage effects. Even
if the latter effects are not yet covered by our theory, the proposed
estimator seems to be quite robust to deviations from the idealised
setting.\vadjust{\goodbreak}


\begin{appendix}\label{app}
\section{From discrete to continuous experiments}\label{appa}
\begin{pf*}{Proof of Theorem~\ref{theo1}}
To establish Le Cam equivalence, we give a constructive proof to
transfer observations in $\mathcal{E}_0$ to the continuous-time model
$\mathcal{E}_1$ and the other way round. We bound the Le Cam distance
by estimates for the squared Hellinger distance between Gaussian
measures and refer to Section {A.1} in \cite{reiss} for
information on Hellinger distances between Gaussian measures and bounds
with the Hilbert--Schmidt norm. The crucial difference here is that
linear interpolation is carried out for nonsynchronous irregular
observation schemes.
Consider the linear B-splines or hat functions
\[
b_{i,n}(t)=\1_{[{(i-1)}/{n}, {(i+1)}/{n}]}(t) \min{ \biggl(1+n \biggl(t-
\frac{i}{n} \biggr), 1-n \biggl(t-\frac{i}{n} \biggr) \biggr)}.
\]
Define $b_i^l(t):=b_{i,n_l}(F_l(t)),1\le i\le n_l,1\le l\le d$, which
are warped spline functions satisfying $b_{i_1}^l(t^{(l)}_{i_2})=\delta
_{i_1,i_2}$. A centered Gaussian process $\hat{Y}$ is derived from
linearly interpolating each component of $Y$:
%
\begin{equation}
\hat{Y}_t^{(l)} = \sum_{i=1}^{n_l}{Y}_{i}^{(l)}
b_i^{l}(t)=\sum_{i=1}^{n_l}X_{t_i^{(l)}}^{(l)}b_i^{l}(t)+
\sum_{i=1}^{n_l}\mathbf {\eps
}^{(l)}_i b_i^{l}(t).
\end{equation}
Setting $A(t)=(a_{lr}(t))_{l,r=1,\ldots,d}=\int_0^t\Sigma(s) \,ds$, the
covariance matrix function $\E[\hat{Y}_t{\hat{Y}_s}^{\top}]$ of the
interpolated process $\hat{Y}$ is determined by
\[
\E \bigl[\hat{Y}^{(l)}_t{\hat{Y}^{(r)}_s}
\bigr]=\sum_{i=1}^{n_l}\sum
_{\nu=1}^{n_r}a_{lr}\bigl(t_{i}^{(l)}
\wedge t_{\nu}^{(r)}\bigr)b_{i}^l(t)b_{\nu
}^r(s)+
\delta_{l,r}\eta_l^2\sum
_{i=1}^{n_l} b_i^l(t)b_i^l(s).
\]
For any $g=(g^{(1)},\ldots,g^{(d)})^{\top} \in L^2([0,1],\mathds
{R}^d)$, we have in the $L^2$-scalar product
\[
\E \bigl[\langle{g},\hat{Y}\rangle^2 \bigr]=\sum
_{l,r=1}^d\sum_{i=1}^{n_l}
\sum_{\nu=1}^{n_r}a_{lr}
\bigl(t_i^{(l)}\wedge t_{\nu
}^{(r)}\bigr)
\bigl\langle g^{(l)},b_i^l\bigr\rangle \bigl
\langle g^{(r)},b_{\nu}^r\bigr\rangle +\sum
_{l=1}^d\sum_{i=1}^{n_l}
\bigl\langle g^{(l)},b_i^l\bigr
\rangle^2\eta_l^2.
\]
The sum of the addends induced by the observation noise in diagonal
terms is bounded from above by
$\sum_{l=1}^d\frac{\eta_l^2}{n_l}\|g^{(l)}/\sqrt{F_l'}\|_{L^2}^2
=\sum_{l=1}^d\|g^{(l)}H_{n,l}\|_{L^2}^2$
since by virtue of $0\le\sum_ib_{i,n}\le1$, $\int b_{i,n}=1/n$ and
Jensen's inequality:
\begin{eqnarray*}
\sum_{i=1}^{n_l}\bigl\langle
g^{(l)},b_i^l\bigr\rangle^2 &\le&
\frac{1}{n_l}\sum_{i=1}^{n_l}\int
_0^1 \bigl(\bigl(g^{(l)}\circ
F_l^{-1}\bigr)\cdot \bigl(F_l^{-1}
\bigr)'\bigr)^2b_{i,n_l}
\\
&\le&\frac{1}{n_l}\int_0^1 \bigl(
\bigl(g^{(l)}\circ F_l^{-1}\bigr)\cdot
\bigl(F_l^{-1}\bigr)' \bigr)^2=
\frac{1}{n_l}\int_0^1 \frac{(g^{(l)})^2}{F_l'}.
\end{eqnarray*}
On the other hand, we have $\E[\langle g,\diag(H_{n,l})_l\, dW\rangle
]=\sum_{l=1}^d
\|g^{(l)}H_{n,l}\|_{L^2}^2$
for a $d$-dimensional standard Brownian motion $W$. Consequently, a
process $\bar{Y}$ with continuous-time\vadjust{\goodbreak} white noise and the same signal
part as $\hat{Y}$ can be obtained by adding uninformative noise.
Introduce the process
%
\begin{equation}
\label{iwn} d\bar{Y}= \Biggl(\sum_{i=1}^{n_l}X_{t_i^{(l)}}b_i^l(t)
\Biggr)_{1\le l\le
d} \,dt +\diag\bigl(H_{n,l}(t)\bigr)_{1\le l\le d}\,dW_t,
\end{equation}
and its associated covariance operator $\bar C\dvtx L^2\rightarrow L^2$,
given by
\[
\bar C g(t)= \Biggl( \sum_{r=1}^d\sum
_{i=1}^{n_l}\sum
_{\nu
=1}^{n_r}a_{lr}\bigl(t_{i}^{(l)}
\wedge t_{\nu}^{(r)}\bigr)\bigl\langle g^{(r)},b_{\nu}^r
\bigr\rangle \Biggr)_{1\le l\le d} + \bigl(H_{n,l}(t)^2g^{(l)}(t)
\bigr)_{1\le l\le d}.
\]
In fact, it is possible to transfer observations from our original
experiment $\mathcal{E}_0$ to observations of \eqref{iwn} by adding
$\mathbf{N}(0,\bar C-\hat C)$-noise, where $\hat C\dvtx L^2\rightarrow L^2$
is the covariance operator of $\hat Y$.
Now, consider the covariance operator
\[
Cg(t)=\int_0^1\int_0^{t\wedge u}A(s)
\,ds\, g(u) \,du + \biggl(\frac{\eta_l^2}{n_lF_l'(t)}g^{(l)}(t) \biggr)_{1\le l\le d},
\]
associated with the continuous-time experiment $\mathcal{E}_1$.

We can bound $C^{-1/2}$ on $L^2([0,1],\R^d)$ from below (by partial
ordering of operators) by a simple matrix multiplication operator:
$C^{-1/2}\le{\cal M}_{\diag(H_{n,l}(t))_l}$.
Denote the Hilbert--Schmidt or Frobenius norm by $\|
\cdot\|_{\mathrm{HS}}$.
The asymptotic equivalence of observing $\bar{Y}$ and $Y$ in $\mathcal
{E}_1$ is ensured by the Hellinger distance bound
\begin{eqnarray*}
&&\operatorname{H}^2 \bigl(\mathcal{L} (\bar{Y} ), \mathcal{L} (Y )
\bigr)\\
&&\qquad\le2 \bigl\|C^{-{1}/{2}}(\bar C-C)C^{-{1}/{2}}\bigr\|^2_{\mathrm{HS}}
\\
&&\qquad\le 2 \int_0^1\int_0^1
\Biggl(\sum_{l=1}^d\sum
_{r=1}^d H_{n,l}(t)^{-2}H_{n,r}(t)^{-2}
\\
&&\quad\qquad\hspace*{40pt}{} \times\Biggl(\sum_{i=1}^{n_l}\sum
_{\nu=1}^{n_r}a_{lr}\bigl(t_{i}^{(l)}
\wedge t_{\nu
}^{(r)}\bigr)b_{i}^l(t)b_{\nu}^r(s)-a_{lr}(t
\wedge s) \Biggr)^2 \Biggr)\,dt \,ds
\\
&&\qquad= 2 \int_0^1\int_0^1
\Biggl(\sum_{l=1}^d\sum
_{r=1}^d\frac{n_ln_r}{\eta_l^2\eta
_r^2}
\\
&&\hspace*{40pt}\qquad\quad{}\times \Biggl(\sum_{i=1}^{n_l}\sum
_{\nu=1}^{n_r}a_{lr}\bigl(t_{i}^{(l)}
\wedge t_{\nu}^{(r)}\bigr)b_{i,n_l}(u)b_{\nu,n_r}(z)\\
&&\hspace*{140pt}{}-a_{lr}
\bigl(F_l^{-1}(u)\wedge F_r^{-1}(z)
\bigr) \Biggr)^2 \Biggr) \,du \,dz
\\
&&\qquad={\mathcal O} \Biggl(R^4\sum_{l=1}^d
\sum_{r=1}^d\eta_l^{-2}
\eta _r^{-2}n_ln_rn_{\mathrm{min}}^{-2-2\beta}
\Biggr).
\end{eqnarray*}
The estimate for the $L^2$-distance between the function $(t,s)\mapsto
A(F_l^{-1}(t)\wedge F_r^{-1}(s)),(l,r)\in\{1,\ldots,d\}^2$, and its
coordinate-wise linear interpolation by $\mathcal{O}(n_{\mathrm{min}}^{-1-\beta
}\vee n_{\mathrm{min}}^{-3/2})$ relies on a standard approximation result on a
rectangular grid of maximal width $(n_{\mathrm{min}})^{-1}$ based on the fact
that this function lies in the Sobolev class $H^{1+\beta}([0,1]^2)$
with corresponding norm bounded by $2R^4$. This follows immediately by
the product rule from $A'=\Sigma\in H^\beta$ and $(F_l^{-1})'\in
C^\beta
$, together with an $L^2$-error bound at the skewed diagonal $\{
(t,s)\dvtx F_l(t)=F_r(s)\}$.

Next, we explicitly show that $\mathcal{E}_1$ is at least as
informative as $\mathcal{E}_0$. To this end, we discretise in each
component on the intervals $I_{i,l}=[\frac{i}{n_l}-\frac
{1}{2n_l},\frac
{i}{n_l}+\frac{1}{2n_l}]\cap[0,1]$ for $i=0,\ldots,n_l$.
Define
%
\begin{eqnarray}\qquad
\bigl(Y_i^{\prime} \bigr)^{(l)}&=&\frac{1}{|I_{i,l}|}
\int_{F_l^{-1}(I_{i,l})}F_l^{\prime}(t)
\,dY^{(l)}_{t} =\frac
{1}{|I_{i,l}|}\int_{F_l^{-1}(I_{i,l})}X_t^{(l)}F_l^{\prime}(t)
\,dt+\epsilon_i^{(l)}
\nonumber
\\[-8pt]
\\[-8pt]
\nonumber
&=&\frac{1}{|I_{i,l}|}\int_{I_{i,l}}X_{F^{-1}(u)}^{(l)}
\,du+\epsilon _i^{(l)},
\end{eqnarray}
%
with i.i.d. $\mathbf{N}(0,\eta_l^2)$-random variables $\epsilon
_i^{(l)}=\frac
{1}{|I_{i,l}|}\int_{F_l^{-1}(I_{i,l})}\eta_l(F_l^{\prime}/n_l)^{1/2}
\,dW_t^{(l)}$.
The covariances are calculated as
\[
\E \bigl[\bigl(Y_i^{\prime}\bigr)^{(l)}
\bigl(Y_{\nu
}^{\prime}\bigr)^{(r)} \bigr] =
\frac{1}{|I_{i,l}||I_{\nu,r}|} \int_{I_{i,l}} \int_{I_{\nu,r}}
a_{lr}\bigl(F_l^{-1}(u) \wedge
F_r^{-1}\bigl(u'\bigr)\bigr)
\,du\,du' + \delta _{l,r}\delta_{i,\nu}
\eta_l^2.
\]
We obtain for the squared Hellinger distance between the laws of observation
\begin{eqnarray*}
&&\operatorname{H}^2 \bigl(\mathcal{L} \bigl(\bigl(Y_i^{(l)}
\bigr)_{l=1,\ldots,d;i=0,\ldots,n_l} \bigr), \mathcal{L} \bigl(\bigl(\bigl(Y_i^{\prime}
\bigr)^{(l)}\bigr)_{l=1,\ldots,d;i=0,\ldots,n_l} \bigr) \bigr)
\\
&&\qquad\le \sum_{l,r=1}^d\eta_l^{-2}
\eta_r^{-2}\sum_{i=0}^{n_l}
\sum_{\nu=0}^{n_r} \biggl(\frac{1}{|I_{i,l}||I_{\nu,r}|}
\int_{I_{i,l}} \int_{I_{\nu,r}} a_{lr}
\bigl(F_l^{-1}(u)\wedge F_r^{-1}
\bigl(u'\bigr)\bigr)\\
&&\hspace*{170pt}{}-a_{lr}\bigl(F_l^{-1}(i/n_l
\wedge\nu/n_r)\bigr) \,du\,du' \biggr)^2.
\end{eqnarray*}
Write $A^F_{lr}(u,u')=a_{lr}(F_l^{-1}(u)\wedge F_r^{-1}(u'))$ and note
$A^F_{lr}\in H^{1+\beta}([0,1]^2)$ due to $A'=\Sigma\in H^\beta$ and
$F_l^{-1},F_r^{-1}\in C^\beta$.
For $(i,\nu)\notin{\cal C}:=\{(0,0),(0,n_r),(n_l,0),$ $(n_l,n_r)\}
$ the rectangle $I_{i,l}\times I_{\nu,r}$ is symmetric around\vadjust{\goodbreak}
$(i/n_l,\nu/n_r)$ such that the integral in the preceding display
equals ($\nabla$ denotes the gradient)
\begin{eqnarray*}
&&\int_{I_{i,l}\times I_{\nu,r}} \int_0^1
\biggl( \biggl\langle\nabla A^F_{lr}\biggl(\frac
{i}{n_l}+
\theta \biggl(u-\frac{i}{n_l}\biggr),\frac
{\nu
}{n_r}+\theta
\biggl(u'-\frac{\nu}{n_r}\biggr) \biggr), \\[-1pt]
&&\hspace*{160pt}\qquad{}\biggl(u-
\frac{i}{n_l},u'-\frac{\nu}{n_r}\biggr)\biggr\rangle
\\[-1pt]
&&\hspace*{72pt}\qquad{}-
\biggl\langle\nabla A^F_{lr}\biggl(
\frac{i}{n_l},\frac{\nu
}{n_r}\biggr), \biggl(u-\frac{i}{n_l},u'-
\frac{\nu}{n_r}\biggr)\biggr\rangle \biggr) \,d\theta \,du\,du'.
\end{eqnarray*}
Using Jensen's inequality, we thus obtain further the bound for the
squared Hellinger distance:
\begin{eqnarray*}
&&\hspace*{-4pt}\sum_{l,r=1}^d\eta_l^{-2}
\eta_r^{-2}\sum_{i=0}^{n_l}
\sum_{\nu=0}^{n_r} \frac{(n_l\vee n_r)^{-2}}{|I_{i,l}||I_{\nu,r}|}\\[-1pt]
&&\hspace*{-6pt}\qquad{}\times\int
_{I_{i,l}\times
I_{\nu,r}} \int_0^1\bigl\|\nabla
A^F_{lr}\bigl(i/n_l+\theta(u-i/n_l),
 \nu/n_r+\theta\bigl(u'-\nu/n_r\bigr)
\bigr)\\[-1pt]
&&\hspace*{153pt}{}-\nabla A^F_{lr}(i/n_l,
\nu/n_r)\1 \bigl((i,\nu )\notin{\cal C}\bigr)\bigr\|^2 \,d\theta
\,du\,du'
\\[-1pt]
&&\hspace*{-6pt}\qquad=\sum_{l,r=1}^d \eta_l^{-2}
\eta_r^{-2}\frac
{n_ln_r}{(n_l\vee n_r)^{2}}\mathcal{O}
\bigl(R^4(n_l\wedge n_r)^{-2\beta
}
\bigr) \\[-1pt]
&&\hspace*{-6pt}\qquad= \mathcal{O} \Biggl( R^4 \Biggl(\sum
_{l=1}^dn_l/\eta_l^2
\Biggr)^2 n_{\mathrm{min}}^{-2-2\beta} \Biggr),
\end{eqnarray*}
where the order estimate is due to $\|\nabla A^F_{lr}\|_{H^\beta}\le
R^2$ and a standard $L^2$-approxi\-mation result for Sobolev spaces,
observing that for the four corner rectangles in $\cal C$ the
boundedness of the respective integrals only adds the total order
$4n_{\mathrm{min}}^{-2}< n_ln_rn_{\mathrm{min}}^{-2-2\beta}$.
\end{pf*}\vspace*{-7pt}

\section{Asymptotics in the block-wise constant experiment}\label{appb}\vspace*{-3pt}

\begin{pf*}{Proof of Theorem~\ref{cltorlle}}
As we have seen, the estimator is unbiased in $\mathcal{E}_2$. For the
covariance structure we use the independence between blocks and
frequencies and the commutativity with $\cal Z$ to infer
%
\begin{eqnarray}\label{variance}
&&\Cov_{\mathcal{E}_2} \bigl({\mathbf I}_n^{1/2}
\operatorname{LMM}_{\mathrm{or}}^{(n)} \bigr)\nonumber\\[-1pt]
&&\qquad= {\mathbf
I}_n^{1/2}\sum_{k=0}^{h^{-1}-1}h^2
\sum_{j=1}^{\infty}W_{jk}
\Cov_{\mathcal{E}_2} \bigl(\operatorname{vec} \bigl(S_{jk}S_{jk}^{\top}
\bigr) \bigr)W_{jk}^\top{\mathbf I}_n^{1/2}
\\[-1pt]
&&\qquad= {\mathbf I}_n^{1/2}\sum
_{k=0}^{h^{-1}-1}h^2I_k^{-1}{
\mathbf I}_n^{1/2}{\cal Z}={\cal Z}.\nonumber
\end{eqnarray}
Since the local Fisher-type information matrices are strictly positive definite,
and thus invertible by Assumption~\ref{smoothness}(iii), the
multivariate CLT \eqref{fclt} for the oracle estimator follows by
applying a standard CLT for triangular schemes as Theorem~4.12 from
\cite{kallenberg}. The Lindeberg condition is implied by the stronger
Lyapunov condition which is easily verified here by bounding moments of
order $4$.

In Appendix \ref{AppC} below, we prove that in experiment ${\mathcal
E}_1$ the estimator $\operatorname{LMM}_{\mathrm{or}}^{(n)}$ has an additional
bias of order $\mathcal{O}(n_{\mathrm{min}}^{-\alpha/2})+\mathcal{O}_P(h)$
and a
difference in the covariance of order $\mathcal{O}(hn_{\mathrm{min}}^{-\alpha
/2})+\mathcal{O}_P(h^2)$ under our Assumption~\ref
{smoothness}(ii-$\alpha$), (iii-$\underline\Sigma$), which by
Slutsky's lemma yields an asymptotically negligible term compared to
the best attainable rate (in any entry) $n_{\mathrm{max}}^{-1/4}$; cf.
Theorem~\ref{ThmCR}.
\end{pf*}

\begin{pf*}{Proof of Corollary~\ref{corclt}}
An important property of our oracle estimator is its equi-variance with
respect to invertible linear transformations $A_k$ on each block $k$ in
the sense that for observed statistics $\tilde S_{jk}:=A_kS_{jk}\sim
\mathbf{N}(0,\tilde C_{jk})$ under $\mathcal{E}_2$ we obtain
[$A^{-\top
}:=(A^\top)^{-1}$ for short]
\begin{eqnarray*}
C_{jk}&=&A_k^{-1}\tilde C_{jk}A_k^{-\top},\qquad
I_{jk}=(A_k\otimes A_k)^\top
\tilde I_{jk} (A_k\otimes A_k),\\
I_k&=&(A_k\otimes A_k)^\top\tilde
I_{k} (A_k\otimes A_k)
\end{eqnarray*}
and hence with some (deterministic) bias correction terms
$B_{jk},\tilde B_{jk}$
\begin{eqnarray*}
\operatorname{LMM}_{\mathrm{or}}^{(n)}&=&\sum
_{k=0}^{h^{-1}-1} h(A_k\otimes
A_k)^{-1} \tilde I_k^{-1} \sum
_{j\ge0} \tilde I_{jk}(A_k
\otimes A_k) \operatorname{vec}\bigl(S_{jk}S_{jk}^\top
-B_{jk}\bigr)
\\
&=&\sum_{k=0}^{h^{-1}-1} (A_k\otimes
A_k)^{-1} \biggl(h\tilde I_k^{-1}
\sum_{j\ge0} \tilde I_{jk} \operatorname{vec}\bigl(\tilde
S_{jk}\tilde S_{jk}^\top -\tilde
B_{jk}\bigr) \biggr).
\end{eqnarray*}
For the covariance, we use commutativity with $\cal Z$ and obtain likewise
%
\begin{equation}
\label{Covtilde} \Cov_{\mathcal
{E}_2}\bigl(\operatorname{LMM}_{\mathrm{or}}^{(n)}
\bigr)=\sum_{k=0}^{h^{-1}-1} h^2(A_k
\otimes A_k)^{-1}\tilde I_k^{-1}(A_k
\otimes A_k)^{-\top}{\cal Z}.
\end{equation}
We use this property to diagonalise the problem on each block.
In terms of the noise level matrix ${\cal H}_k:=\diag
(H_{l,n}^k)_{l=1,\ldots,d}$, let $O_k$ be an orthogonal matrix such that
%
\begin{equation}
\label{pca2} \Lambda^{kh}=O_k\mathcal{H}_k^{-1}
\Sigma^{kh}\mathcal {H}_k^{-1}O_k^\top
\end{equation}
is diagonal. Note that $\Lambda^{kh}$ grows with $n$, but we drop the
dependence on $n$ in the notation for all matrices $\Lambda^{kh}$,
$O_k$ and ${\cal H}_k$. Use $A_k=O_k\mathcal{H}_k^{-1}$ to obtain the
spectral statistics \eqref{spec} transformed:
\[
\tilde S_{jk}=O_k\mathcal{H}_k^{-1}S_{jk}
\sim\mathbf{N} (\mathbf {0},\tilde C_{jk} )\qquad \mbox{independent for all }
(j,k),
\]
which yields a simple-structured diagonal covariance matrix:
\[
\tilde C_{jk}=O_k\mathcal{H}_k^{-1}C_{jk}
\mathcal{H}_k^{-1}O_k^{\top} =
\Lambda^{kh}+\frac{\pi^2j^2}{h^2}E_d.
\]
A key point is that the covariance structure \eqref{Covtilde} in $\R
^{d^2\times d^2}$ is for independent components $\tilde S_{jk}$ also
diagonal, up to symmetry in the co-volatility\break  matrix entries. 
Summing\vadjust{\goodbreak}
$\tilde I_{jk}$ over $j$ is explicitly solvable and gives\break  for
\mbox{$p,q=1,\ldots,d$}
\begin{eqnarray*}
\bigl(h\tilde I_k^{-1}\bigr)_{p,q}&=&
\Biggl(h^{-1}\sum_{j=1}^\infty\bigl(
\tilde C_{jk}^{-1}\otimes\tilde C_{jk}^{-1}
\bigr)_{p,q} \Biggr)^{-1}
\\
&=& \Biggl(h^{-1}\sum_{j=1}^\infty
\bigl(\Lambda_{pp}^{kh}+\pi^2j^2h^{-2}
\bigr)^{-1} \bigl(\Lambda_{qq}^{kh}+
\pi^2j^2h^{-2}\bigr)^{-1}
\Biggr)^{-1}
\\
&=& \biggl( \frac{\sqrt{\Lambda_{qq}^{kh}}\coth(h\sqrt{\Lambda_{pp}^{kh}})
-\sqrt{\Lambda_{pp}^{kh}}\coth(h\sqrt{\Lambda_{qq}^{kh}})}{
2\sqrt{\Lambda_{pp}^{kh}\Lambda_{qq}^{kh}} (\Lambda
_{qq}^{kh}-\Lambda
_{pp}^{kh})} - \frac{1}{2h\Lambda_{pp}^{kh}\Lambda_{qq}^{kh}} \biggr)^{-1}
\\
&=& 2\Bigl(\Lambda_{pp}^{kh}\sqrt{
\Lambda_{qq}^{kh}}+\Lambda _{qq}^{kh}
\sqrt {\Lambda_{pp}^{kh}}\Bigr) \\
&&{}\times\bigl(1+\mathcal{O}
\bigl(e^{-2h\sqrt{\Lambda_{pp}^{kh}\wedge\Lambda
_{qq}^{kh}}}+h^{-1}\bigl(\Lambda_{pp}^{kh}
\wedge \Lambda_{qq}^{kh}\bigr)^{-1/2} \bigr) \bigr),
\end{eqnarray*}
using $\Lambda^{kh}\ge(\min_{l,t}n_lF_l'(t)\eta_l^{-2})\underline
\Sigma
\gtrsim n_{\mathrm{min}}E_d$, $h^2n_{\mathrm{min}}\to\infty$ and $\coth(x)=1+\mathcal
{O}(e^{-2x})$ for $x\to\infty$. We thus obtain uniformly over $k$
\[
h\tilde I_k^{-1}=\bigl(2+\KLEINO(1)\bigr) \bigl(
\Lambda^{kh}\otimes\sqrt{\Lambda^{kh}} + \sqrt{
\Lambda^{kh}} \otimes\Lambda^{kh}\bigr).
\]
By formula \eqref{Covtilde}, we infer in terms of $(\Sigma_{\cal
H}^{kh})^{1/2}:={\cal H}_k({\cal H}_k^{-1}\Sigma^{kh}{\cal
H}_k^{-1})^{1/2}{\cal H}_k$
\[
\Cov_{\mathcal{E}_2}\bigl(\operatorname{LMM}_{\mathrm{or}}^{(n)}
\bigr)=\bigl(2+\KLEINO (1)\bigr)\sum_{k=0}^{h^{-1}-1}
h\bigl(\Sigma^{kh}\otimes\bigl(\Sigma_{\cal H}^{kh}
\bigr)^{1/2} + \bigl(\Sigma_{\cal H}^{kh}
\bigr)^{1/2} \otimes\Sigma^{kh}\bigr) {\cal Z}.
\]
The final step consists in combining $n_{\mathrm{min}}^{1/2}H_{n,l}(t)\to
H_l(t)$ uniformly in $t$ together with a Riemann sum approximation to conclude
\begin{eqnarray*}
&&\lim_{n_{\mathrm{min}}\to\infty}n_{\mathrm{min}}^{1/2}
\Cov_{\mathcal
{E}_2}\bigl(\operatorname {LMM}_{\mathrm{or}}^{(n)}\bigr)
\\
&&\qquad =2 \biggl(\int_0^1 \bigl(\Sigma\otimes
\bigl({\cal H}\bigl({\cal H}^{-1}\Sigma{\cal H}^{-1}
\bigr)^{1/2}{\cal H}\bigr)\\
&&\hspace*{59pt}{} + \bigl({\cal H}\bigl({\cal H}^{-1}
\Sigma{\cal H}^{-1}\bigr)^{1/2}{\cal H}\bigr) \otimes\Sigma
\bigr) (t) \,dt \biggr) {\cal Z}.
\end{eqnarray*}
\upqed\end{pf*}

\section{Proofs for continuous models}\label{AppC}

\subsection{Weight matrix estimates}

We shall often need general norm bounds on the weight matrices $W_{jk}$.

\begin{lem}\label{LemWjk}
$\!\!\!$The oracle weight matrices satisfy $\Vert W_{jk} \Vert\lesssim
h_0^{-1}(1+j^4/h_0^4)^{-1}$ uniformly over $(j,k)$ and matrices $\Sigma
^{kh}$ with $\Vert\Sigma^{kh} \Vert_\infty+\Vert(\Sigma
^{kh})^{-1} \Vert_\infty
\lesssim1$.
\end{lem}

\begin{pf}
From the proof of Corollary~\ref{corclt}, we infer
\[
W_{jk}=\bigl(H_kO_k^\top\otimes
H_kO_k^\top\bigr)\tilde W_{jk}\bigl(O_k
H_k^{-1}\otimes O_k H_k^{-1}\bigr)\vadjust{\goodbreak}
\]
with
\[\tilde W_{jk}=\bigl(2+\KLEINO(1)\bigr)h^{-1} \bigl(\bigl(
\Lambda^{kh}\tilde C_{jk}^{-1}\bigr)\otimes\bigl(
\sqrt{\Lambda^{kh}}\tilde C_{jk}^{-1}\bigr)+\bigl(
\sqrt {\Lambda^{kh}}\tilde C_{jk}^{-1}\bigr)\otimes
\bigl(\Lambda^{kh}\tilde C_{jk}^{-1}\bigr) \bigr).
\]
We evaluate one factor in $W_{jk}$ using
\[
\bigl\Vert H_kO_k^\top\Lambda^{kh}
\tilde C_{jk}^{-1}O_k H_k^{-1}
\bigr\Vert =\bigl\Vert\Sigma^{kh}\bigl(\Sigma^{kh}+
\pi^2j^2h^{-2}H_k^2
\bigr)^{-1} \bigr\Vert \lesssim \bigl(1+j^2h^{-2}n_{\mathrm{min}}^{-2}
\bigr)^{-1}.
\]
By $\Vert A\otimes B \Vert\le\Vert A \Vert\Vert B \Vert$ and
$\sqrt{\Lambda^{kh}}\tilde C_{jk}^{-1}=(\Lambda^{kh}\tilde
C_{jk}^{-1})(\Lambda^{kh})^{-1/2}$ (the matrices are diagonal), we infer
$\Vert W_{jk} \Vert\lesssim h^{-1}(1+j^2h_0^{-2})^{-2}
\Vert H_kO_k^\top(\Lambda^{kh})^{-1/2} O_k H_k^{-1} \Vert$.
To evaluate the last norm, despite matrix multiplication is
noncommutative, we note
\begin{eqnarray*}
\bigl(O_k^\top\bigl(\Lambda^{kh}
\bigr)^{-1/2} O_k H_k^{-1}
\bigr)^\top O_k^\top\bigl(\Lambda^{kh}
\bigr)^{-1/2} O_k H_k^{-1}&=&
H_k^{-1}O_k^\top \bigl(\Lambda
^{kh}\bigr)^{-1}O_kH_k^{-1}
\\
&=& \bigl(\Sigma ^{kh}\bigr)^{-1},
\end{eqnarray*}
whence by polar decomposition $\vert O_k^\top(\Lambda^{kh})^{-1/2}
O_k H_k^{-1} \vert=(\Sigma^{kh})^{-1/2}$ implies
\[
\bigl\Vert O_k^\top\bigl(\Lambda^{kh}
\bigr)^{-1/2} O_k H_k^{-1}\bigr \Vert=\bigl\Vert
\bigl(\Sigma ^{kh}\bigr)^{-1/2} \bigr\Vert\lesssim1.
\]
Together with $\Vert H_k \Vert\lesssim n_{\mathrm{min}}^{-1/2}$ this yields
$\Vert W_{jk} \Vert\lesssim h^{-1}(1+j^2h_0^{-2})^{-2}n_{\mathrm{min}}^{-1/2}$,
which gives the result.
\end{pf}

Moreover, for the adaptive estimator we have to control the dependence
of the weight matrices $W_{jk}=W_j(\Sigma^{kh})$ on $\Sigma^{kh}$.
We use the notion of matrix differentiation as introduced in \cite
{fackler}: define the derivative $dA/dB$ of a matrix-valued function
$A(B)\in\mathds{R}^{o\times p}$ with respect to $B\in\mathds
{R}^{q\times r}$ as the $\mathds{R}^{op\times qr}$ matrix with row
vectors $(d/dB_{ab})\operatorname{vec}(A),1\le a\le q,1\le b\le r$.

\begin{lem}\label{derivative}
For the derivatives of the oracle weight matrices $W_{j}(\Sigma^{kh})$,
assuming $\|\Sigma^{kh}\|_{\infty}+\|(\Sigma^{kh})^{-1}\|_{\infty
}\lesssim1$, we have uniformly over $(j,k)$:
%
\begin{equation}
\biggl\llVert \frac{d}{d\Sigma^{kh}}W_{j}\bigl(\Sigma ^{kh}
\bigr)\biggr\rrVert \lesssim h_0^{-1}\bigl(1+j^4h_0^{-4}
\bigr)^{-1}.
\end{equation}
\end{lem}
\begin{pf}
Since the notion of matrix derivatives relies on vectorisation, the
identities $\operatorname{vec}(I_k^{-1}I_{jk})=(E_{d^2}\otimes
I_k^{-1})\operatorname{vec}(I_{jk})=(I_{jk}^{\top}\otimes E_{d^2})\operatorname{vec}(I_k^{-1})$
give rise to the matrix differentiation product rule
%
\begin{equation}
\label{pr}\frac{d}{d\Sigma^{kh}} W_{jk}= (I_{jk}\otimes
E_{d^2} )\frac{dI_k^{-1}}{d\Sigma^{kh}}+ \bigl(E_{d^2}\otimes
I_k^{-1} \bigr)\frac{dI_{jk}}{d\Sigma^{kh}}.
\end{equation}
Applying the mixed product rule $(A\otimes B)(C\otimes D)=(AC\otimes
BD)$ repeatedly, and the differentiation product rule and chain rule to
$I_{jk}= C_{jk}^{-1}\otimes C_{jk}^{-1}$, we obtain
\begin{eqnarray*}
&&\frac{d}{d C_{jk}} \bigl( C_{jk}^{-1}\otimes
C_{jk}^{-1} \bigr)\\[-2pt]
&&\qquad =- \bigl( \bigl( C_{jk}^{-1}
\otimes C_{jk}^{-1} \bigr)\otimes \bigl( C_{jk}^{-1}
\otimes C_{jk}^{-1} \bigr) \bigr)\\[-2pt]
&&\qquad\quad{}\times \bigl( \bigl((
C_{jk}\otimes E_d \otimes E_{d^2})
+(E_{d^2}\otimes E_{d}\otimes C_{jk})
\bigr) (E_d\otimes C_{d,d}\otimes E_d) \\[-2pt]
&&\hspace*{100pt}\qquad\quad{}\times\bigl(
\bigl( \operatorname{vec}(E_{d})\otimes E_{d^2}\bigr)+ \bigl(E_{d^2}
\otimes \operatorname{vec}(E_{d})\bigr) \bigr) \bigr),
\end{eqnarray*}
with the so-called commutation matrix $C_{d,d}=\mathcal{Z}-E_{d^2}$. By
orthogonality of the last factors in both addends, $\|A\otimes B\|=\|A\|
\|B\|$, and the mixed product rule, we infer for the norm of the second
addend in \eqref{pr}
\begin{eqnarray*}
\biggl\llVert \bigl(E_{d^2}\otimes I_k^{-1}\bigr)
\frac
{dI_{jk}}{d\Sigma
^{kh}}\biggr\rrVert &\le&2 \bigl\llVert \bigl(E_d\otimes
C_{jk}^{-1} \bigr)\otimes \bigl(I_k^{-1}
\bigl( C_{jk}^{-1}\otimes C_{jk}^{-1}
\bigr) \bigr)\bigr\rrVert
\\[-2pt]
&=& 2 \llVert W_{jk}\rrVert \bigl\llVert C_{jk}^{-1}
\bigr\rrVert \lesssim\llVert W_{jk}\rrVert.
\end{eqnarray*}
By virtue of
$(I_k^{-1}\otimes E_{d^2})\frac{dI_k}{d\Sigma^{kh}}=-(E_{d^2}\otimes
I_k)\frac{dI_k^{-1}}{d\Sigma^{kh}}$
it follows with the mixed product rule that $dI_k^{-1}/d\Sigma
^{kh}=-(I_k^{-1}\otimes I_k^{-1})(dI_k/d\Sigma^{kh})$. This yields for
the norm of the first addend in \eqref{pr}
\begin{eqnarray*}
\biggl\|(I_{jk}\otimes E_{d^2})\frac
{dI_k^{-1}}{d\Sigma
^{kh}} \biggr\|&=& \biggl\|
\bigl(W_{jk}^{\top}\otimes I_k^{-1}
\bigr)\frac
{dI_k}{d\Sigma
^{kh}} \biggr\|\lesssim\|W_{jk}\| \biggl\|\bigl(E_{d^2}
\otimes I_k^{-1}\bigr)\sum_{j^{\prime}}
\frac{dI_{j^{\prime}k}}{d\Sigma^{kh}} \biggr\|
\\[-2pt]
&\lesssim&\|W_{jk}\| \biggl(\sum_{j^{\prime}}
\|W_{j^{\prime}k}\| \biggr) \lesssim\|W_{jk}\|
\end{eqnarray*}
since we can differentiate inside the sum by the absolute convergence
of $\sum_{j^{\prime}}\|W_{j^{\prime}k}\|$. This proves our claim by
Lemma~\ref{LemWjk}.\vspace*{-2pt}
\end{pf}

\subsection{Bias bound}\label{SecBias}

Using the formula $1-2\sin^2(x)=\cos(2x)$ and It\^{o} isometry, the
$(d\times d)$-matrix of (negative) biases (in the signal) of the
addends in \eqref{orlle}
as an estimator of $\Sigma^{kh}$ in experiment $\mathcal{E}_1$ is
given by
\[
B_{j,k}:=2h^{-1} \int_{kh}^{(k+1)h}
\Sigma (t)\cos\bigl(2j\pi h^{-1}(t-kh)\bigr) \,dt, 
\]
which has the structure of a $j$th Fourier cosine coefficient. We
introduce the corresponding weighting function in the time domain:
\[
G_k(u)=2\sum_{j=1}^\infty
W_{jk}\cos(2j\pi u)\in\R^{d^2\times d^2},\qquad u\in[0,1].
\]
Parseval's identity then shows for the $d^2$-dimensional block-wise
bias vector of~\eqref{orlle}:
\[
\sum_{j=1}^\infty W_{jk}\operatorname{vec}(B_{j,k})
=
\int_{kh}^{(k+1)h} h^{-1}G_k
\bigl(h^{-1}(t-kh)\bigr)\operatorname{vec}\bigl(\Sigma(t)\bigr) \,dt.\vadjust{\goodbreak}
\]
The vector of total biases of \eqref{orlle} is then the linear
functional of $\Sigma$:
\[
\sum_{k=0}^{h^{-1}-1}h\sum
_{j=1}^\infty W_{jk}\operatorname{vec}(B_{jk}) =
\int_0^1 G^h(t)\operatorname{vec}\bigl(\Sigma(t)
\bigr) \,dt,
\]
where for $t\in[kh,(k+1)h)$
\[
G^h(t)=G_k\bigl(h^{-1}(t-kh)\bigr)=2\sum
_{j=1}^\infty W_{jk}\cos\bigl(2\pi
jh^{-1} t\bigr).
\]

For $\Sigma$ in the Besov space $B^{\alpha}_{1,\infty}([0,1])$,
$0<\alpha\le1$, the $L^1$-modulus of continuity satisfies $\omega
_{L^1([0,1])}(\Sigma,\delta)\le\Vert\Sigma \Vert_{B^\alpha
_{1,\infty
}}\delta
^\alpha$; see, for example, \cite{Cohen}, Section~3.2. We have for
$\delta\in(0,1)$ and $s\in[0,1-\delta]$
\begin{eqnarray*}
&&{\biggl\vert\int_0^{\delta} \operatorname{vec}\bigl(\Sigma(t+s)\bigr)
\cos\biggl(\frac{2\pi
t}{\delta }\biggr) \,dt\biggr \vert}\\
&&\qquad =\frac{1}{\delta}\biggl\llvert
\int_{0}^{\delta} \int_0^\delta
\operatorname{vec}\bigl(\Sigma (t+s)-\Sigma(u+s)\bigr)
 \,du\cos\biggl(\frac{2\pi t}{\delta}\biggr) \,dt\biggr\rrvert \\
 &&\qquad\le\sup
_{0\le
v\le\delta
}\int_0^\delta\bigl\vert \operatorname{vec}
\bigl(\Sigma(t+s)-\Sigma(t+v+s)\bigr) \bigr\vert \,dt \le
\omega _{L^1([s,s+\delta])}(\Sigma,
\delta).
\end{eqnarray*}
This shows for the total bias in estimation of the volatility in $X$ by
the bound on $\Vert W_{jk} \Vert$ in Lemma~\ref{LemWjk}
\begin{eqnarray*}
{\biggl\vert\int_0^1G^h(t)\operatorname{vec}\bigl(
\Sigma(t)\bigr) \,dt\biggr \vert} &\le&2\sum_{k=0}^{h^{-1}-1}
\sum_{j=1}^\infty\Vert W_{jk}
\Vert\omega_{L^1([kh,(k+1)h])}(\Sigma,h/j)
\\
&\lesssim&\sum_{j=1}^\infty
h_0^{-1}\bigl(1+(h_0/j)^4
\bigr)^{-1}(h/j)^\alpha \asymp(h/h_0)^\alpha=n_{\mathrm{min}}^{-\alpha/2}.
\end{eqnarray*}
We thus have a bias of order $\mathcal{O}(n_{\mathrm{min}}^{-\alpha/2})$. Remark
that it is quite surprising that this bias bound is independent of $h$,
which is also at the heart of the quasi-maximum likelihood method \cite
{aitfanxiu2010}.

If $\operatorname{vec}(\Sigma)$ is a (vector-valued) square-integrable martingale,
then we use that martingale differences are uncorrelated and write for
the total bias
\[
\int_0^1 G^h(t)\operatorname{vec}\bigl(
\Sigma(t)\bigr) \,dt=\int_0^1
G^h(t)\operatorname{vec}\bigl(\Sigma (t)-\Sigma \bigl(\bigl\lfloor
h^{-1}t \bigr\rfloor h\bigr)\bigr) \,dt,
\]
using $\int G_k=0$. This expression is centred with covariance matrix
\begin{eqnarray*}
&&\sum_{k=0}^{h^{-1}-1} \int_{[kh,(k+1)h]^2}G_k
\bigl(h^{-1}(t-kh)\bigr)\E \bigl[\operatorname{vec}\bigl(\Sigma(t)-\Sigma(kh)\bigr) \operatorname{vec}
\bigl(\Sigma(s)-\Sigma(kh)\bigr)^\top\bigr]\\
&&\hspace*{58pt}\qquad{}\times G_k\bigl(h^{-1}(s-kh)\bigr) \,dt\,ds.
\end{eqnarray*}
The expected value in the display is smaller than (in matrix ordering)
$\E[\operatorname{vec}(\Sigma((k+1)h)-\Sigma(kh))\operatorname{vec}(\Sigma((k+1)h)-\Sigma
(kh))^\top
]$. Because of $\Vert G_k \Vert_\infty\lesssim1$ the covariance
matrix (in
any norm) is of order $\mathcal{O}(h^2\E[\Vert\Sigma(1)-\Sigma (0)
\Vert^2])=\mathcal{O}(h^2)$.

If $\Sigma=\Sigma^B+\Sigma^M$ is the sum of a function $\Sigma^B$ in
$B^\alpha_{1,\infty}([0,1])$ and a square-integrable martingale
$\Sigma
^M$, then the preceding estimations apply for each summand and the
total bias has maximal order $\mathcal{O}(n_{\mathrm{min}}^{-\alpha
/2})+\mathcal
{O}_P(h)$.

\subsection{Variance for general continuous-time model}

The covariance for the estimator under model $\mathcal{E}_1$ can be
calculated as under model $\mathcal{E}_2$, but we lose independence
between different frequencies $j,j'$ on the same block. For that, we
use the formula for Gaussian random vectors $A,B$
\begin{eqnarray*}
&&\Cov\bigl(\operatorname{vec}\bigl(AA^\top\bigr),\operatorname{vec}
\bigl(BB^\top\bigr)
\bigr) \\
&&\qquad= \bigl(\Cov(B,B)\otimes\Cov(A,B)+\Cov(A,A)\otimes
\Cov(A,B)\\
&&\qquad\quad{}
+\Cov(A,B)\otimes\Cov(A,A) +\Cov(A,B)\otimes\Cov(B,B) \bigr){\cal
Z}/4,
\end{eqnarray*}
obtained by polarisation. This implies
\begin{eqnarray*}
&&\bigl\Vert\Cov_{\mathcal{E}_1}\bigl(\operatorname {LMM}_{\mathrm{or}}^{(n)}
\bigr) -\Cov_{\mathcal{E}_2}\bigl(\operatorname {LMM}_{\mathrm{or}}^{(n)}
\bigr) \bigr\Vert
\\
&&\qquad\lesssim\sum_{k=0}^{h^{-1}-1}h^2
\sum_{j,j'=1}^\infty\Vert W_{j'k} \Vert
\bigl\Vert W_{jk}\bigl(\Cov_{\mathcal{E}_1}(S_{jk},S_{jk})
\otimes\Cov _{\mathcal {E}_1}(S_{jk},S_{j'k})\bigr) \bigr\Vert.
\end{eqnarray*}
From Lemma~\ref{LemWjk} and $\Vert A\otimes B \Vert\le\Vert A \Vert
\Vert B \Vert$ for
matrices $A,B$, we infer that the series over $j,j'$ is bounded in
order by
\begin{eqnarray*}
&&\sum_{j,j'=1}^\infty h_0^{-2}
\bigl(1+j'/h_0\bigr)^{-4}(1+j/h_0)^{-2}\\
&&\qquad{}\times\biggl( \biggl\| \int_0^1 (\Sigma-\bar
\Sigma_h) (t)\frac{\Phi_{jk}(t)\Phi
_{j'k}(t)} {\Vert\Phi_{jk} \Vert_{L^2}\Vert\Phi_{j'k} \Vert
_{L^2}} \,dt \biggr\|\\
&&\hspace*{6pt}\qquad\quad{}+ \biggl\|\int_0^1\diag\bigl(H^2_{n,l}-
\overline{H}^2_{n,l,h}\bigr) (t) \phi _{jk}(t)
\phi_{j'k}(t) \,dt \biggr\| \biggr).
\end{eqnarray*}
The identities $2\cos(a)\cos(b)=\cos(a+b)+\cos(a-b)$, $2\sin
(a)\sin
(b)=\cos(a-b)-\cos(a+b)$ and the same bound as in Section~\ref{SecBias}
imply for $\Sigma, (F_1')^{-1},\ldots,\break (F_d')^{-1}\in B^\alpha
_{1,\infty
}([0,1])$ [note that even $(F_l')^{-1}\in C^\alpha([0,1])$]
\begin{eqnarray*}
&&\biggl\| \int_0^1 (\Sigma-\bar
\Sigma_h) (t)\frac{\Phi_{jk}(t)\Phi
_{j'k}(t)} {\Vert\Phi_{jk} \Vert_{L^2}\Vert\Phi_{j'k}
 \Vert_{L^2}} \,dt\biggr \|\\
&&\qquad\lesssim h^{-1}
\biggl(\frac{h}{j+j'}+\frac{h(1-\delta
_{j,j'})}{\vert j-j' \vert} \biggr)^\alpha \Vert\Sigma \Vert_{B^\alpha
_{1,\infty}([kh,(k+1)h])}
\end{eqnarray*}
and similarly the bound
\[
h^{-1}\biggl(\frac{h}{j+j'}+\frac{h(1-\delta
_{j,j'})}{\vert j-j' \vert}\biggr)^\alpha jj'h_0^{-2}\max_l\bigl\Vert
\bigl(F_l'\bigr)^{-1} \bigr\Vert_{B^\alpha_{1,\infty}([kh,(k+1)h])}
\]
for the norm over $H^2_{n,l}$. Putting all estimates together gives
\begin{eqnarray*}
&&\bigl\Vert\Cov_{\mathcal{E}_1}\bigl(\operatorname {LMM}_{\mathrm{or}}^{(n)}
\bigr) -\Cov_{\mathcal{E}_2}\bigl(\operatorname {LMM}_{\mathrm{or}}^{(n)}
\bigr)\bigr \Vert
\\
&&\qquad\lesssim h\sum_{j,j'=1}^\infty
h_0^{-2}\bigl(1+j'/h_0
\bigr)^{-4}(1+j/h_0)^{-2}h^\alpha
\bigl(1+\bigl\vert j-j' \bigr\vert \bigr)^{-\alpha}\bigl(1
+jj'h_0^{-2}\bigr).
\end{eqnarray*}
By comparison with $\int_0^\infty\int_0^\infty
(1+y)^{-4}(1+x)^{-2}\vert x-y \vert^{-\alpha}(1+xy) \,dx \,dy \lesssim1
$ (in terms of $x\approx j/h_0$, $y\approx j'/h_0$)
we conclude
\[
\bigl\Vert\Cov_{\mathcal{E}_1}\bigl(\operatorname{LMM}_{\mathrm{or}}^{(n)}
\bigr) -\Cov _{\mathcal{E}_2}\bigl(\operatorname{LMM}_{\mathrm{or}}^{(n)}
\bigr)\bigr \Vert\lesssim h n_{\mathrm{min}}^{-\alpha/2}.
\]
Arguing exactly as in Section~\ref{SecBias} for the case of $\Sigma$
being a sum of a $B^\alpha_{1,\infty}$-function and an
$L^2$-martingale, the difference of covariances is in general of order
$\mathcal{O}(h n_{\mathrm{min}}^{-\alpha/2})+\mathcal{O}_P(h^2)$.

\subsection{Proof of Theorem \texorpdfstring{\protect\ref{cltadlle}}{4.4}}\label{App:Adaptive}

Let us denote the rate of convergence of $\hat\Sigma$ by $\delta
_n=n_{\mathrm{min}}^{-\alpha/(4\alpha+2)}$. For later use, we note the order bounds
%
\begin{equation}
\label{Eqdelta_n} \qquad\delta_n=\KLEINO \bigl(r^{1/2}h_0^{-1/2}(n_{\mathrm{min}}/n_{\mathrm{max}})^{1/4}
\bigr), \qquad\delta_n=\KLEINO \bigl(h_0^{-1}(n_{\mathrm{min}}/n_{\mathrm{max}})^{1/2}
\bigr).
\end{equation}
First, we show that
%
\begin{equation}
\label{EqoP} \bigl\|\operatorname{LMM}_{\mathrm{or}}^{(n)}-
\operatorname{LMM}_{\mathrm{ad}}^{(n)}\bigr\| =\KLEINO_P
\bigl(n_{\mathrm{max}}^{-1/4}\bigr),
\end{equation}
which by Slutsky's lemma implies the CLT with normalisation matrix
$\mathbf{I}_n$. This in turn is already sufficient for obtaining the
result of Corollary~\ref{corclt} for $\operatorname{LMM}_{\mathrm{ad}}^{(n)}$.
Let us start with proving that
\[
T_n^m:= \Biggl\|\sum_{m=0}^{r^{-1}-1}h
\sum_{k=mr/h}^{(m+1)r/h-1}\sum
_{j=1}^\infty \bigl(W_{j}\bigl(\hat\Sigma
^{mr}\bigr)-W_{j}\bigl(\Sigma^{mr}\bigr) \bigr)
Z_{jk} \Biggr\|=\KLEINO_P\bigl(n_{\mathrm{max}}^{-1/4}
\bigr),
\]
where the random variables
\[
Z_{jk}=\operatorname{vec} \bigl(S_{jk}S_{jk}^{\top}-
\pi^2j^2h^{-2}\diag {\bigl(\bigl(H_{n,l}^{kh}
\bigr)^2\bigr)}_{1\le l\le d}-\Sigma^{kh} \bigr)
\]
are independent, $\E_{\mathcal{E}_2}[Z_{jk}]=0$, $\Cov_{\mathcal
{E}_2} (Z_{jk} )=I_{jk}^{-1}\mathcal{Z}$.
We have
%
\begin{equation}
\label{tight1} T_n^m\le\sum_{m=0}^{r^{-1}-1}h
\sum_{j=1}^\infty\bigl\llVert
W_{j}\bigl(\hat\Sigma^{mr}\bigr)-W_{j}\bigl(
\Sigma^{mr}\bigr)\bigr\rrVert \Biggl\|\sum_{k=mr/h}^{(m+1)r/h-1}Z_{jk}
\Biggr\|,
\end{equation}
since the weight matrices do not depend on $k$ on the same block of the
coarse grid.
Using Lemma~\ref{derivative} and that $\|\hat\Sigma-\Sigma\|
_{L^1}=\mathcal{O}_P(\delta_n)$, we obtain
\begin{eqnarray*}
\bigl\llVert W_j\bigl(\hat\Sigma^{mr}
\bigr)-W_j\bigl(\Sigma^{mr}\bigr)\bigr\rrVert &\le&\max
_k \biggl\|\frac{dW_{j}(\Sigma^{kh})}{d\Sigma^{kh}} \biggr\|\bigl \Vert\hat\Sigma^{mr}-
\Sigma^{mr} \bigr\Vert
\\
&=&\mathcal{O}_P \bigl( \bigl(h_0^{-1}
\wedge h_0^3j^{-4} \bigr)r^{-1}\|
\hat\Sigma-\Sigma\|_{L^1([mr,(m+1)r])} \bigr).
\end{eqnarray*}
For the second factor in \eqref{tight1}, we employ $\|\Cov_{\mathcal
{E}_2}(Z_{jk})\|=2\|C_{jk}\|^2$. Consequently, \eqref{Eqdelta_n}
implies for $T_n^m$ the bound
\begin{eqnarray*}
&&\sum_{m=0}^{r^{-1}-1}h\bigl\|\hat
\Sigma^{mr}-\Sigma^{mr}\bigr\|\sum_{j=1}^\infty
\mathcal{O} \bigl(\bigl(h_0^{-1}\wedge
h_0^3j^{-4}\bigr) \bigl(1\vee
j^2h_0^{-2}\bigr) \bigr)
\\
&&\qquad=\|\hat\Sigma-\Sigma\|_{L^1([0,1])} \|\mathcal{O} \bigl(r^{-1/2}h^{1/2}
\bigr)=\mathcal{O}_P \bigl(r^{-1/2}h^{1/2}
\delta_n \bigr)=\KLEINO_P\bigl(n_{\mathrm{max}}
^{-1/4}
\bigr).
\end{eqnarray*}

The asymptotics \eqref{EqoP} follow if we can ensure that the coarse
grid approximations of the weights induce a negligible error, that is,
if also
\[
\sum_{m=0}^{r^{-1}-1}\sum
_{k=mr/h}^{(m+1)r/h-1}h\sum_{j=1}^\infty
\bigl(W_{j}\bigl(\Sigma^{kh}\bigr)-W_{j}\bigl(
\Sigma^{mr}\bigr) \bigr) Z_{jk}=\KLEINO _P
\bigl(n_{\mathrm{max}}^{-1/4}\bigr)
\]
holds. The term is centred and its covariance matrix is bounded in norm by
\[
\sum_{m=0}^{r^{-1}-1}\sum
_{k=mr/h}^{(m+1)r/h-1}h^2\sum
_{j=1}^\infty \bigl\llVert W_{j}\bigl(
\Sigma^{kh}\bigr)-W_{j}\bigl(\Sigma^{mr}\bigr)
\bigr\rrVert ^2 \bigl\| I_{jk}^{-1}\bigr\|.
\]
From Lemma~\ref{derivative}, $\Vert I_{jk}^{-1} \Vert=2\Vert C_{jk}
\Vert^2\lesssim1+j^4h_0^{-4}$ and $\Sigma\in B^\alpha_{1,\infty
}([0,1])$ we derive the upper bound
\[
\mathcal{O} \Biggl(\sum_{k=0}^{h^{-1}-1}h^2
\sum_{j=1}^\infty r^2
h_0^{-2} \bigl(1+j^4h_0^{-4}
\bigr)^{-1} \Biggr) =\mathcal{O} \bigl(n_{\mathrm{min}}^{-1/2}r^{2\alpha}
\bigr)=\KLEINO \bigl(n_{\mathrm{max}}^{-1/2}\bigr)
\]
by the choice of $r$ and $\alpha>1/2$.

Another application of Slutsky's lemma yields the CLT with
normalisation matrix $\hat{\mathbf{I}}_n$ provided $\mathbf
{I}_n^{1/2}\hat{\mathbf{I}}_n^{-1/2}\to E_{d^2}$ in probability. The
proof of Lemma~\ref{derivative}, more specifically the bound on the
last term in \eqref{pr}, yields also
\[
{\biggl\Vert\frac{d}{d\Sigma^{kh}}I_j\bigl(\Sigma^{kh}\bigr)
\biggr\Vert}\lesssim h_0^{-1}\bigl(1+j^4h_0^{-4}
\bigr)^{-1}.
\]
This implies $\sum_{k,j}\Vert\hat I_{jk}-I_{jk} \Vert=\mathcal
{O}_P(h^{-1}\delta_n)$.
Using $\hat A^{-1}-A^{-1}=A^{-1}(\hat A-A)\hat A^{-1}$ and $\Vert
I_k^{-1} \Vert\lesssim h_0^{-1}$, we infer
\[
\bigl\Vert\hat{\mathbf{I}}_n^{-1}-\mathbf{I}_n^{-1}
\bigr\Vert\le\sum_{k=0}^{h^{-1}-1} h^2 {
\Biggl\Vert \Biggl(\sum_{j=1}^\infty\hat
I_{jk} \Biggr)^{-1}- \Biggl(\sum
_{j=1}^\infty I_{jk} \Biggr)^{-1}
\Biggr\Vert}=\mathcal {O}_P \bigl(h\delta_nh_0^{-2}
\bigr).
\]

The smallest eigenvalue of $\mathbf{I}_n^{-1}$ equals $\Vert\mathbf
{I}_n \Vert^{-1}$ which has order at least $n_{\mathrm{max}}^{-1/2}$. The global
Lipschitz constant $L_n$ of $f(x)=x^{1/2}$ for $x\ge\Vert\mathbf
{I}_n \Vert^{-1}$ is therefore of order $n_{\mathrm{max}}^{1/4}$. The perturbation
result from \cite{Kittaneh} for functional calculus therefore implies
\[
\bigl\Vert\mathbf{I}_n^{1/2}\hat{\mathbf{I}}_n^{-1/2}-E_d
\bigr\Vert \le L_n\bigl\Vert\mathbf{I}_n^{1/2} \bigr\Vert\bigl\Vert
\mathbf{I}_n^{-1}-\hat {\mathbf{I}}_n^{-1}
\bigr\Vert =\mathcal{O}_P \bigl(n_{\mathrm{max}}^{1/2}h
\delta_nh_0^{-2} \bigr).
\]
The order is $(n_{\mathrm{max}}/n_{\mathrm{min}})^{1/2}h_0^{-1}\delta_n$ and tends to
zero by \eqref{Eqdelta_n}.
\end{appendix}

\begin{supplement}[id=suppA]
\stitle{Lower bound proofs for estimating the quadratic covariation
matrix from noisy observations}
\slink[doi]{10.1214/14-AOS1224SUPP} 
\sdatatype{.pdf}
\sfilename{aos1224\_supp.pdf}
\sdescription{We provide detailed proofs for Section~\ref{sec:4}.}
\end{supplement}

%
%

%


\printaddresses
\end{document}